\documentclass[11pt,letterpaper]{article}
\usepackage{amsmath,amssymb,amsthm,mathtools,bm}
\usepackage{bbm}
\usepackage[textwidth=6.2in,textheight=9in,centering]{geometry}
\usepackage{microtype,booktabs,array,graphicx,enumitem}
\usepackage[numbers,sort&compress]{natbib}
\usepackage{pgfplots}
\usepgfplotslibrary{groupplots}
\usepackage[hidelinks]{hyperref}
\usepackage{comment}
\hypersetup{pdftitle={Local Maxima of Discrete Gaussian Processes: Exact Distributions and Infill Limits},pdfauthor={Dan Cheng and John Ginos}}

\setlist[enumerate]{itemsep=.3em}
\numberwithin{equation}{section}
\newtheorem{proposition}{Proposition}[section]
\newtheorem{theorem}[proposition]{Theorem}
\newtheorem{lemma}[proposition]{Lemma}
\newtheorem{corollary}[proposition]{Corollary}
\theoremstyle{definition}

\newtheorem{example}[proposition]{Example}
\theoremstyle{remark}
\newtheorem{remark}[proposition]{Remark}
\newcommand{\1}{\mathbbm{1}}
\newcommand{\Var}{\operatorname{Var}}
\newcommand{\Cov}{\operatorname{Cov}}
\newcommand{\Corr}{\operatorname{Corr}}
\newcommand{\dd}{\,\mathrm{d}}
\newcommand{\E}{\mathbb{E}}
\renewcommand{\P}{\mathbb{P}}
\newcommand{\R}{\mathbb{R}}
\newcommand{\Z}{\mathbb{Z}}

\begin{document}
	
	\title {Local Maxima of Discrete Gaussian Processes: Exact Distributions and Infill Limits}
	\author{Dan Cheng \\ Arizona State University \and John Ginos \\ Arizona State University }
	
	\date{}
	
	\maketitle
	
	\begin{abstract}
		We study local maxima of centered, unit-variance Gaussian processes
		observed on one-dimensional grids. For general nonstationary covariance,
		we give exact sitewise peak probabilities, conditional height densities,
		and expected peak counts using the covariance of a site and its two
		neighbors. The height limits established here are unified by a continuous family indexed by the magnitude of the conditional height-curvature correlation parameter $\eta\in[0,1]$, with a Gaussian endpoint at $\eta=0$, and
		smooth peak laws for $0<\eta\le 1$ with a Rayleigh endpoint at $\eta=1$.
		Under continuous covariance derivatives through total order four, we prove convergence to the smooth
		family and obtain corresponding intensity and expected count limits.
		Additional sample path and Kac--Rice regularity assumptions identify
		these limits with the intensity and Palm height law of continuous local
		maxima. For general nonstationary covariance, a necessary and sufficient
		increment criterion characterizes convergence to the Gaussian endpoint,
		including when peak probabilities vanish. Explicit finite-grid bounds
		control density and tail errors and yield uniform approximations across
		sites. For each fixed nonconstant, mean-square continuous stationary
		process, the fourth spectral moment gives a complete classification:
		finiteness yields a smooth family member, whereas an infinite moment yields the Gaussian endpoint and divergent expected peak counts on every interval of positive length. The height laws converge in total variation in both regimes. Nonstationary examples, Mat\'ern covariance models, and numerical comparisons illustrate the results, including logarithmic boundary behavior and the distinction between distributional convergence and relative tail accuracy.
	\end{abstract}
	
	
	
	\section{Introduction}\label{sec:introduction}
	
	Local maxima provide natural candidates for detecting signals in temporal
	and spatial data. Applications include neuroimaging, genomics, change point
	analysis, and environmental signal detection. Selecting a point because it
	is a local maximum changes its null height distribution: a marginally
	Gaussian process need not have Gaussian heights at its selected peaks.
	Conditional peak height distributions therefore provide a building block
	for inference at local maxima
	\citep{SchwartzmanGavrilovAdler11,CS17,CS19sphere,ChangePoint}.
	
	For smooth Gaussian processes and fields, expected numbers and height
	distributions of local maxima are studied through Kac--Rice formulas
	\citep{Rice45,CL67,AT07,AzaisWschebor09}.
	Explicit peak height distributions and their dependence on covariance
	structure have been developed in
	\citep{CS15,CS18,Cheng25}.
	There is an essential distinction between continuous and discrete
	parameter spaces. Under the usual nondegeneracy and smoothness conditions,
	a fixed deterministic point of a continuous process is a local maximum
	with probability zero. The height distribution of a continuous peak is
	then defined through local Palm sampling and calculated by Kac--Rice.
	On a discrete grid, the event that a specified site exceeds its two
	neighbors has positive probability under the finite-dimensional
	nondegeneracy conditions, so its height has an ordinary conditional law.
	
	The discrete law is directly relevant when data are recorded at finitely
	many locations and peaks are identified by comparisons between neighboring
	observations. A continuous approximation may be accurate for a sufficiently
	fine grid and a sufficiently smooth process, but its accuracy depends on
	both the covariance and the mesh. Discrete local maximum methods for
	bounding excursion probabilities were developed by
	\citep{Worsley05,TaylorWorsleyGosselin07}.
	In particular, Taylor et al.
	\citep{TaylorWorsleyGosselin07} calculate the conditional
	probability that neighboring observations fall below a given central
	height and integrate this probability against the Gaussian density. In
	one dimension, that calculation is equivalent to the stationary joint
	height and peak probability used here. Lin et al.
	\citep{LinSchwartzmanDavenport25} explicitly normalize such
	probabilities to obtain finite-lattice peak height laws and develop a Monte
	Carlo method for general stationary neighborhood covariance matrices.
	
	The present paper complements these finite-lattice calculations with
	nonstationary formulas and rigorous infill limits. We establish convergence
	to a curvature-weighted density under fourth-order covariance smoothness,
	give a necessary and sufficient nonstationary criterion for a Gaussian
	height limit, and classify stationary limits through the fourth spectral
	moment. Remark~\ref{remark:unify} organizes the resulting height laws in
	the continuous family $\{h_\eta^{\mathrm{SM}}:0\le\eta\le1\}$, whose
	endpoints are Gaussian and Rayleigh. The finite-dimensional Gaussian
	conditioning calculation is an established ingredient; our principal
	results concern how discrete peak selection and covariance regularity
	determine the limiting law and the accuracy of its finite-grid
	approximation.

Let $X=\{X(t):t\in I\}$ be a centered, unit-variance, real-valued Gaussian process on an
interval $I\subset\R$, with covariance
\[
C(s,t)=\Cov\{X(s),X(t)\}.
\]
In the stationary case, we write $C(s,t)=r(s-t)$. Under stationarity, since the quantities below do not depend on location, we omit their subscript $t$; for the spectral results, $X$ is understood to be defined on $\mathbb{R}$ and observed only on $I$.

For a mesh size $d>0$ and a site whose two neighbors lie in $I$, write the event of observing a local maximum at $t$ as
\[
\mathcal A_{t,d}
=\{X(t)>X(t-d),\ X(t)>X(t+d)\}.
\]
We study the sitewise peak probability $p_{t,d}=\P(\mathcal A_{t,d})$ and
the conditional peak height density
\[
h_{t,d}(x)
=\frac{\mathrm d}{\mathrm dx}
\P\{X(t)\le x\mid\mathcal A_{t,d}\}.
\]

The main contributions are as follows.
\begin{enumerate}
	\item We express $p_{t,d}$ and $h_{t,d}$ using a bivariate Gaussian orthant
	probability and a bivariate Gaussian distribution function. The formulas
	require only the covariance matrix of $\{X(t-d),X(t),X(t+d)\}$ and no
	sample path differentiability. Summing the sitewise peak probabilities
	and integrating the joint height densities gives exact expected counts
	of all grid peaks and of grid peaks above a threshold.
	
	\item When $C$ has four derivatives near the diagonal, we prove that
	$p_{t,d}/d$ converges to a curvature-weighted intensity and $h_{t,d}$
	converges to $h_{\eta_t}^{\mathrm{SM}}$, with $\eta_t\in(0,1]$;
	see Theorem~\ref{thm:smooth-limit}. The height law depends on the single
	parameter $\eta_t$, the magnitude of the conditional height-curvature
	correlation. The endpoint $\eta_t=1$ has zero conditional curvature
	variance and a Rayleigh height law. Under additional sample path and
	Kac--Rice regularity assumptions, these limits describe the intensity
	and Palm height distribution of continuous local maxima.
	
	\item For general nonstationary covariance, let
	$v_{t,d}^{\pm}=1-C(t,t\pm d)$ and let $q_{t,d}$ be the correlation
	between $X(t)-X(t-d)$ and $X(t)-X(t+d)$. We prove that
	\[
	R_{t,d}=\frac{v_{t,d}^{-}+v_{t,d}^{+}}{1+q_{t,d}}\to0
	\]
	is necessary and sufficient for weak convergence of peak heights to
	$N(0,1)$, and is equivalent to pointwise and $L^1$ convergence of their
	densities. The criterion allows the peak probability to vanish and
	does not require convergence of $q_{t,d}$. Explicit finite-grid bounds
	give uniform sitewise approximations and control relative tail errors
	over a specified threshold range. In the stationary case,
	$R_d=4\{1-r(d)\}^2/[4\{1-r(d)\}-\{1-r(2d)\}]$.
	The criterion includes nondifferentiable processes and processes with
	one, but not two, mean-square derivatives.
	
	\item For each fixed nonconstant, mean-square continuous stationary
	process with spectral moments $m_j$, we give a complete classification.
	If $m_4<\infty$, the limiting density is $h_\eta^{\mathrm{SM}}$ with
	$\eta=m_2/\sqrt{m_4}\in(0,1]$. If $m_4=\infty$, the limit is the
	Gaussian endpoint $h_0^{\mathrm{SM}}=\phi$, and expected peak counts
	diverge on every interval of positive length. The Mat\'ern family
	illustrates this boundary, including its logarithmic cases.
\end{enumerate}

The unification in Remark~\ref{remark:unify} gives a common description
of the limiting height laws established in this paper. The smooth density $h_{\eta}^{\mathrm{SM}}$ extends continuously to the Gaussian law as $\eta\downarrow0$ and to the Rayleigh law
as $\eta\uparrow1$, with convergence in total variation at both endpoints.
In the smooth regime, $\eta$ measures the relation between height and
curvature conditional on zero slope; the Gaussian endpoint describes
asymptotically height-neutral peak selection. For a fixed stationary
process, the spectral classification specifies the member that arises:
$\eta=0$ when $m_4=\infty$, and $\eta=m_2/\sqrt{m_4}>0$ when
$m_4<\infty$. The interpolation of the family therefore describes
variation across covariance models, while the infill theorems identify
the limits under their respective assumptions.

The difference between the smooth and rough regimes is fundamental.  For a
smooth process, the probability that a specified grid point is a local
maximum is of order $d$; summing over approximately $d^{-1}$ sites produces a
finite limiting number of continuous local maxima. The comparison event concentrates near zero slope, with a width proportional to negative curvature, and produces the curvature weighting in $h_{\eta_t}^{\mathrm{SM}}$ which is generally non-Gaussian. For a rough process, grid peaks become increasingly numerous as the grid is refined.  Whether $t$ is a discrete local maximum is determined by the increments from $X(t)$ to its neighboring grid values. As the grid spacing tends to zero, these increments become asymptotically independent of $X(t)$, so the local maximum event carries essentially no information about the height at $t$. This explains why the marginal Gaussian density reappears after conditioning.

The rest of the paper is organized as follows.  Section~\ref{sec:exact}
derives the exact finite-mesh distributions and expected count formulas.
Section~\ref{sec:smooth} proves the smooth infill limit and its stationary
specialization. Section~\ref{sec:rough} gives the nonstationary Gaussian-limit criterion,
the unified family interpretation, and the stationary spectral
classification.  Section~\ref{sec:examples} treats common covariance models, with stationary and nonstationary examples and numerical comparisons. Section~\ref{sec:discussion} discusses
approximation choices and the scope of the inference implications.

\section{Exact Discrete Distributions}\label{sec:exact}

\subsection{Grid peaks and notation}

Let \(I=[a,b]\) and, for \(d>0\), define
\[
\mathcal G_d=I\cap(a+d\mathbb Z),\qquad
\mathcal G_d^\circ
=\{t\in\mathcal G_d:t-d,t+d\in\mathcal G_d\}.
\]
We restrict attention to grid points whose two neighboring grid points also lie in $I$. Boundary points, which require a separate one-sided definition of a local maximum, are not considered here.
Throughout, $X$ is centered and $\Var\{X(t)\}=1$.  Write $\phi$, $\Phi$, and $\Psi=1-\Phi$ for the standard Gaussian
density, distribution function, and survival function. For
$\rho\in[-1,1]$, let $\Phi_2(x,y;\rho)$ be the distribution function
of a standard bivariate Gaussian vector with correlation $\rho$.
This definition includes the singular cases:
\begin{equation}\label{eq:singular-cdf}
	\Phi_2(x,y;1)=\Phi(\min\{x,y\}),\qquad
	\Phi_2(x,y;-1)=[\Phi(x)-\Phi(-y)]_+,
\end{equation}
where $z_+=\max\{z,0\}$.

For $t\in\mathcal G_d^\circ$, define the event of local maximum at $t$ and the corresponding probability as
\begin{equation}\label{A-def}
	\mathcal A_{t,d}
	=\{X(t)>X(t-d),\ X(t)>X(t+d)\},
	\qquad p_{t,d}=\P(\mathcal A_{t,d}).
\end{equation}
Whenever $p_{t,d}>0$, define respectively the cumulative distribution function and the density function of the height of a local maximum at $t$ as
\begin{align*}
	H_{t,d}(u)
	=\P\{X(t)\le u\mid\mathcal A_{t,d}\},\qquad
	h_{t,d}(u)=H_{t,d}'(u).
\end{align*}
Denote by $\overline H_{t,d}(u)=1-H_{t,d}(u)$ the tail of peak height distribution.

\subsection{The exact nonstationary formula}

For compact notation, set
\[
c_-=C(t,t-d),\quad c_+=C(t,t+d),\quad
c_0=C(t-d,t+d).
\]
When $c_\pm<1$, the two comparison differences $X(t)-X(t-d)$ and $X(t)-X(t+d)$ have positive variances $2(1-c_\pm)$, so ties between the center and either neighbor have probability zero. Their correlation is
\begin{equation}\label{eq:q-general}
	q_{t,d}
	=\Corr\{X(t)-X(t-d),X(t)-X(t+d)\}
	=\frac{1-c_--c_++c_0}{2\sqrt{(1-c_-)(1-c_+)}}.
\end{equation}
For $c_\pm\in (-1, 1)$, also let
\begin{align}\label{eq:xi-general}
	\xi_{t,d}
	=\frac{c_0-c_-c_+}
	{\sqrt{(1-c_-^2)(1-c_+^2)}},\quad 
	b_{t,d}^{-}
	=\sqrt{\frac{1-c_-}{1+c_-}},\quad
	b_{t,d}^{+}
	=\sqrt{\frac{1-c_+}{1+c_+}}.
\end{align}

\begin{theorem}[Exact sitewise peak law]\label{thm:exact}
	  Suppose $c_\pm\in(-1,1)$. Then $q_{t,d}\in(-1,1]$ and the peak height density at $t$ is given by
	\begin{equation}\label{eq:h-exact}
		h_{t,d}(x)=\frac{\phi(x)\Phi_2\!\left(
			b_{t,d}^{-}x,b_{t,d}^{+}x;\xi_{t,d}
			\right)}{p_{t,d}}, \quad p_{t,d}
			=\frac14+\frac{1}{2\pi}\arcsin(q_{t,d})>0, \quad  x\in\R.
	\end{equation}
\end{theorem}

\begin{proof}
	Let $Y_-=X(t)-X(t-d)$ and $Y_+=X(t)-X(t+d)$. Since $X$ is
	centered and has unit variance,
	\[
	\Var(Y_\pm)=2(1-c_\pm)>0,\qquad
	\Cov\{X(t),Y_\pm\}=1-c_\pm>0.
	\]
	Define $Z_\pm=Y_\pm/\sqrt{2(1-c_\pm)}$. These are standard Gaussian
	variables with correlation $q_{t,d}$, and
	\[
	\Cov\{X(t),Z_\pm\}=\sqrt{\frac{1-c_\pm}{2}}>0.
	\]
	By Cauchy--Schwarz, together with $\Var\{X(t)\}=1$, we obtain
	\[
	2(1+q_{t,d})=\Var(Z_-+Z_+)
	\geq
	\left\{\sqrt{\frac{1-c_-}{2}}+
	\sqrt{\frac{1-c_+}{2}}\right\}^{\!2}>0.
	\]
	Hence $q_{t,d}>-1$. The bivariate Gaussian orthant identity
	\citep{Barnett:1991}, including the singular case $q_{t,d}=1$, gives
	\[
	p_{t,d}=\P(Y_->0,Y_+>0)
	=\frac14+\frac{1}{2\pi}\arcsin(q_{t,d})>0.
	\]
	
	The conditional vector $(X(t-d),X(t+d)\mid X(t)=x)$ has mean $(c_-x,c_+x)$ and covariance
	\[
	\begin{pmatrix}
		1-c_-^2&c_0-c_-c_+\\
		c_0-c_-c_+&1-c_+^2
	\end{pmatrix}.
	\]
	Both conditional marginal variances are positive, and their
	standardized correlation is $\xi_{t,d}\in[-1,1]$.  On defining
	\[
	Z_1 =\frac{X(t-d)-c_-x}{\sqrt{1-c_-^2}},\qquad
	Z_2 =\frac{X(t+d)-c_+x}{\sqrt{1-c_+^2}},
	\]
	the conditional vector $(Z_1,Z_2 \mid X(t)=x)$ is standard bivariate Gaussian with
	correlation $\xi_{t,d}$.  
	Hence
	\begin{equation}\label{eq:A-cond}
		\begin{split}
	\P(\mathcal A_{t,d}\mid X(t)=x) &= \P(X(t-d)<x, X(t+d)<x \mid X(t)=x)\\
	&= \P(Z_1<b_{t,d}^{-}x, Z_2<b_{t,d}^{+}x \mid X(t)=x)\\
	&=\Phi_2\!\left(b_{t,d}^{-}x,b_{t,d}^{+}x;\xi_{t,d}\right).
	\end{split}
	\end{equation}
	This argument remains valid when the conditional vector is singular. It follows that
	\begin{align*}
		H_{t,d}(u)
		=\frac{1}{p_{t,d}}\P\{\mathcal A_{t,d}, X(t)\le u\}= \frac{1}{p_{t,d}}\int_{-\infty}^u \P\{\mathcal A_{t,d}\mid  X(t)=x\}\phi(x)\dd x.
	\end{align*}
	The integrand is continuous, including at $\xi_{t,d}=\pm1$. Plugging \eqref{eq:A-cond} and taking differentiation proves
	\eqref{eq:h-exact}.
\end{proof}

	Define the joint height and peak-event density by
	\begin{equation}\label{eq:g-exact}
		g_{t,d}(x)=\phi(x)\Phi_2(b_{t,d}^-x,b_{t,d}^+x;\xi_{t,d})
		=\phi(x)\P(\mathcal A_{t,d}\mid X(t)=x).
	\end{equation}
	Then
	\begin{equation}\label{eq:joint-density}
		\P\{X(t)\in\dd x,\mathcal A_{t,d}\}=g_{t,d}(x)\dd x,
		\qquad \int_{\R}g_{t,d}(x)\dd x=p_{t,d}.
	\end{equation}

\begin{remark}
	\label{remark:non-degenerate}
	The positive definiteness of the covariance matrix of
	$(X(t-d),X(t),X(t+d))$ is sufficient for
	the assumption in Theorem \ref{thm:exact}, but is not necessary. Singular covariance matrices are permitted with the convention \eqref{eq:singular-cdf}.
\end{remark}	

\begin{remark}[Equivalent sign convention]
	\label{remark:gamma}
	If one uses
	\[
	\gamma_{t,d}
	=\Corr\{X(t)-X(t-d),X(t+d)-X(t)\}=-q_{t,d},
	\]
	then $p_{t,d}=\arccos(\gamma_{t,d})/(2\pi)$.  The convention used in \eqref{A-def}
	keeps the local maximum event in the positive orthant.
\end{remark}

\subsection{Expected counts of peaks}

Define the numbers of all interior grid peaks and of peaks above $u$ by
\begin{align*}
	M_d(I)=\sum_{t\in\mathcal G_d^\circ}\1_{\mathcal A_{t,d}},\qquad
	M_{u,d}(I)=\sum_{t\in\mathcal G_d^\circ}
	\1_{\mathcal A_{t,d}}\1_{\{X(t)>u\}}.
\end{align*}
Linearity of expectation, together with \eqref{eq:g-exact} and \eqref{eq:joint-density}, yields
\begin{equation}\label{eq:EMd}
	\E M_d(I)
	=\sum_{t\in\mathcal G_d^\circ}p_{t,d},\qquad
	\E M_{u,d}(I)
	=\sum_{t\in\mathcal G_d^\circ}
	\int_u^\infty g_{t,d}(x)\dd x.
\end{equation}
No independence between sites is required. The bivariate CDF expressions
apply whenever $-1<c_\pm<1$ at every included site. The quantity $p_{t,d}$ in \eqref{eq:h-exact} is a sitewise probability, rather than a count intensity per unit length; its length-normalized counterpart is $p_{t,d}/d$. 

\subsection{Stationary specialization}

Suppose now that $X$ is stationary with correlation
$r(s)=\E\{X(t)X(t+s)\}$.  Let
\begin{equation}\label{eq:q-xi-stationary}
	q_d=\frac{1-2r(d)+r(2d)}{2\{1-r(d)\}},\quad
	\xi_d=\frac{r(2d)-r(d)^2}{1-r(d)^2},\quad
	b_d=\sqrt{\frac{1-r(d)}{1+r(d)}}.
\end{equation}

\begin{corollary}[Exact stationary peak law]\label{cor:stationary-exact}
	Suppose $r(d)\in (-1,1)$.  At every interior site, the peak height density is given by
	\begin{align}\label{eq:h-stationary}
		h_d(x)=\frac{\phi(x)\Phi_2(b_dx,b_dx;\xi_d)}{p_d}, \quad p_d=\frac14+\frac{1}{2\pi}\arcsin(q_d),\quad x\in \R.
	\end{align}
	If the grid has $n_d$ interior sites, then
	\begin{align}\label{eq:EMd-stationary}
		\E M_d&=n_dp_d,\quad
		\E M_{u,d}=n_d\int_u^\infty
		\phi(x)\Phi_2(b_dx,b_dx;\xi_d)\dd x.
	\end{align}
\end{corollary}

\begin{proof}
	Stationarity gives $c_-=c_+=r(d)$ and $c_0=r(2d)$.
	Substitution in \eqref{eq:q-general} and \eqref{eq:xi-general} yields $q_{t,d}=q_d$, $\xi_{t,d}=\xi_d$ and $b_{t,d}^{-}=b_{t,d}^{+}
	=b_d$. Theorem~\ref{thm:exact} therefore gives \eqref{eq:h-stationary}, independently of the interior site $t$. Finally, \eqref{eq:EMd-stationary} follows from \eqref{eq:EMd} with $n_d=|\mathcal G_d^\circ|$.
\end{proof}

\begin{example}[Independent Gaussian observations]\label{ex:iid}
	For an i.i.d. standard Gaussian sequence, $r(d)=r(2d)=0$ at distinct grid
	sites.  Hence $p_d=1/3$ and $h_d(x)=3\phi(x)\Phi(x)^2$.
	This is the density of the maximum of three independent standard Gaussian
	variables. This model is not mean-square continuous and lies outside the infill classification below.
\end{example}

\section{The Infill Limit for Smooth Covariance}\label{sec:smooth}

\subsection{The smooth covariance target and its Palm interpretation}

Fix an interior point $t$.  Assume that $C$ has continuous partial
derivatives through total order four in a neighborhood of $(t,t)$. Then $X$
has the first and second mean-square derivatives there.  Write
\begin{equation}\label{eq:lambdas}
	\begin{split}
	\lambda_1(t)&=\Var\{X'(t)\}>0,\qquad
	\lambda_2(t)=\Var\{X''(t)\},\\
	a(t)&=\Cov\{X'(t),X''(t)\}=\lambda_1'(t)/2,\\
	\delta_t^2
	&=\Var\{X''(t)\mid X'(t)\}=\lambda_2(t)-a(t)^2/\lambda_1(t).
\end{split}
\end{equation}
The unit-variance condition $C(t,t)=1$ implies
$\Cov\{X(t),X'(t)\}=0$ and
$\Cov\{X(t),X''(t)\}=-\lambda_1(t)$. Consequently,
\begin{align}
	\tau_t^2 &=\Var\{X''(t)\mid X(t),X'(t)\}=\delta_t^2-\lambda_1(t)^2\ge 0,
	\label{eq:tau-conditional}\\
	\eta_t &= -{\rm Corr}\{X(t), X''(t) \mid X'(t)\} = \lambda_1(t)/\delta_t\in (0, 1].
	\label{eq:rho}
\end{align}
In particular, the degenerate case $\eta_t=1$ is equivalent to $\tau_t=0$.

Recall $(y)_+=\max\{y,0\}$. By \cite{Cheng25}, for a smooth process, its
peak height density is the ratio of intensities $J_t(x)$ and $J_t$:
\begin{equation}\label{eq:palm-definition}
	h_{\eta_t}^{\mathrm{SM}}(x)
	=J_t(x)/J_t,
\end{equation}
where
\begin{equation}\label{eq:J_tx}
	\begin{split}
	J_t(x)
	&=f_{X(t),X'(t)}(x,0)
	\E\!\left[(-X''(t))_+\mid X(t)=x,X'(t)=0\right]
	\\
	&= \begin{cases*}
		\frac{\delta_t\phi(x)}{\sqrt{2\pi\lambda_1(t)}}\left\{\sqrt{1-\eta_t^2}\,
		\phi\!\left(\frac{\eta_tx}{\sqrt{1-\eta_t^2}}\right)
		+\eta_t x
		\Phi\!\left(\frac{\eta_t x}{\sqrt{1-\eta_t^2}}\right)\right\},& \quad $0<\eta_t<1$,\\
		\frac{\sqrt{\lambda_1(t)}}{2\pi}
		x e^{-x^2/2}\1_{\{x>0\}},& \quad $\eta_t=1$,
	\end{cases*}
	\end{split}
\end{equation}
and 
\begin{align}\label{eq:J_t}
	J_t=f_{X'(t)}(0)
	\E\!\left[(-X''(t))_+\mid X'(t)=0\right]=\frac{\delta_t}{2\pi\sqrt{\lambda_1(t)}}>0.
\end{align}
Integrating $J_t(x)$ over $x$ gives $J_t$, thus the ratio in
\eqref{eq:palm-definition} is a probability density. When $\eta_t=1$, $\delta_t=\lambda_1(t)$ and hence $J_t=\sqrt{\lambda_1(t)}/(2\pi)$. If an a.s. $C^2$ version and the standard Kac--Rice regularity assumptions
hold, then $J_t$ is the local intensity of continuous local maxima and $h_{\eta_t}^{\mathrm{SM}}$ is the local Palm height density
\citep{AT07,AzaisWschebor09,Cheng25}. The discrete infill limits below require only the stated covariance assumptions.

\subsection{A uniform bivariate Gaussian scaling lemma}

The following lemma makes the infill argument uniform even when the
conditional variance $\tau_t^2=0$. As in Section~\ref{sec:exact},
$\Phi_2$ includes singular Gaussian laws at correlation $\xi_{t,d}=-1$.

\begin{lemma}[Uniform bivariate Gaussian scaling limit]\label{lem:antipodal-uniform}
	Let $\Theta$ be an index set. For $\theta\in\Theta$, suppose
	$A_{\theta,d}$, $B_{\theta,d}\in\R$ and
	$\zeta_{\theta,d}\in[-1,1]$ satisfy, uniformly in $\theta$,
	\[
	\frac{A_{\theta,d}}d\to A_\theta,\qquad
	\frac{B_{\theta,d}}d\to B_\theta,\qquad
	\frac{1+\zeta_{\theta,d}}{d^2}\to\beta_\theta,
	\]
	where $A_\theta,B_\theta\in \R$ and $\beta_\theta\geq0$ are uniformly bounded.
	Then
	\begin{equation}\label{eq:near-antipodal-limit}
		\sup_{\theta\in\Theta}\left|
		\frac{\Phi_2(A_{\theta,d},B_{\theta,d};\zeta_{\theta,d})}{d}
		-\phi(0)\E\!\left[(A_\theta+B_\theta
		-\sqrt{2\beta_\theta}\,U)_+\right]\right|\to 0,
	\end{equation}
	where $U\sim N(0,1)$. Moreover, if $|A_{\theta,d}|/d+|B_{\theta,d}|/d\leq L$ and
	$(1+\zeta_{\theta,d})/d^2\leq K$ for all sufficiently small $d$, then
	$\Phi_2(A_{\theta,d},B_{\theta,d};\zeta_{\theta,d})/d\le c(L+\sqrt K)$,
	where $c>0$ is a constant independent of $\theta$ and $d$.
\end{lemma}

\begin{proof}
	Suppress $\theta$ temporarily and let $(Z_1,Z_2)$ be standard bivariate
	Gaussian with correlation $\zeta_d$. The independent Gaussian coordinates
	\[
	S=(Z_1+Z_2)/2=d\kappa U,\qquad
	T=(Z_1-Z_2)/2=\nu V,
	\]
	have
	$\kappa=\sqrt{(1+\zeta_d)/(2d^2)}$ and
	$\nu=\sqrt{(1-\zeta_d)/2}$, where $U,V$ are independent standard
	Gaussian. The assumptions give $\nu\to1$ uniformly, so $\nu$ is bounded
	away from zero for small $d$. This representation also holds when
	$\kappa=0$. Let $a=A_d/d$, $b=B_d/d$ and
	$f_\nu(y)=\nu^{-1}\phi(y/\nu)$. The two inequalities $Z_1\le A_d$ and $Z_2\le B_d$ defining
	$\Phi_2$ are equivalent to
	$d(\kappa U-b)\leq T\leq d(a-\kappa U)$. Using the restriction $\kappa U-b<a-\kappa U$ and making the substitution $y=dz$, we obtain
	\begin{equation}\label{eq:antipodal-interval}
		\frac{\Phi_2(A_d,B_d;\zeta_d)}d
		=\E\!\left[
		\1_{\{2\kappa U<a+b\}}
		\int_{\kappa U-b}^{a-\kappa U}f_\nu(dz)\dd z
		\right].
	\end{equation}
	Endpoint equalities do not affect the integral. Its interval length is
	$\ell=(a+b-2\kappa U)_+$, which is bounded by
	$|a|+|b|+2\kappa|U|$. This proves the asserted probability bound, since
	$\sup_y f_\nu(y)=\phi(0)/\nu$.
	
	For the convergence, boundedness gives a constant $M$ such that
	$|a|,|b|,\kappa\leq M$ uniformly. Every integration point in
	\eqref{eq:antipodal-interval} has $|z|\leq M(1+|U|)$. Thus the
	absolute difference between that expression and $\phi(0)\E\ell$ is
	bounded by
	\[
	\E\!\left[
	2M(1+|U|)\sup_{|z|\leq M(1+|U|)}
	|f_\nu(dz)-\phi(0)|\right].
	\]
	The supremum tends to zero for each fixed $U$, uniformly in $\theta$,
	and is bounded by a constant. The integrable envelope
	$M'(1+|U|)$ for some constant $M'$ therefore proves convergence uniformly in $\theta$. Finally, $y\mapsto(y)_+$ is Lipschitz with constant one, and
	\[
	\left|\E(a+b-2\kappa U)_+
	-\E(A_\theta+B_\theta-\sqrt{2\beta_\theta}\,U)_+\right|
	\leq|a-A_\theta|+|b-B_\theta|
	+|2\kappa-\sqrt{2\beta_\theta}|\,\E|U|.
	\]
	The right-hand side tends to zero uniformly; convergence of the square
	roots holds also at $\beta_\theta=0$. This proves the lemma.
\end{proof}

\subsection{Convergence of probabilities, densities and counts}

\begin{theorem}[Smooth infill limit]\label{thm:smooth-limit}
	Suppose $C$ has continuous partial derivatives through total order four
	in a neighborhood of $(t,t)$, and $\lambda_1(t)>0$. Then, for a fixed interior point $t$, as $d\downarrow0$, $p_{t,d}/d\to J_t$ and $g_{t,d}(x)/d \to J_t(x)$ for every $x\in\R$.
	Consequently, $h_{t,d}(x)$ converges pointwise and in $L^1(\R)$ to the smooth peak height density
	\begin{equation}\label{eq:h-smooth-rho}
			h_{\eta_t}^{\mathrm{SM}}(x)
			=\begin{cases}
			\sqrt{1-\eta_t^2}\,
			\phi\!\left(\frac{x}{\sqrt{1-\eta_t^2}}\right)
			+\sqrt{2\pi}\,\eta_t x\phi(x)
			\Phi\!\left(\frac{\eta_t x}{\sqrt{1-\eta_t^2}}\right),& \quad 0<\eta_t<1,\\
			x e^{-x^2/2}\1_{\{x>0\}}, &\quad \eta_t=1,
		    \end{cases}
	\end{equation}
	and thus the discrete peak height distributions converge in total variation.
\end{theorem}

\begin{proof}
	(i) We first prove the nondegenerate case $0<\eta_t<1$ (i.e., $\tau>0$). For brevity, suppress the argument $t$ and write
	$\lambda_1=\lambda_1(t)$, $\lambda_2=\lambda_2(t)$,
	$a=a(t)$, $\delta=\delta_t$, and $\tau=\tau_t$.  For $i,j\geq0$, let
	\[
	C_{ij}=\left.
	\frac{\partial^{i+j}}{\partial s^i\partial v^j}C(s,v)
	\right|_{s=v=t}.
	\]
	Symmetry gives $C_{ij}=C_{ji}$. Differentiating $C(s,s)=1$ along the
	diagonal yields
	\begin{equation}\label{eq:diagonal-derivative-identities}
		C_{01}=0,\qquad C_{02}=-C_{11},\qquad
		C_{03}=-3C_{12},\qquad
		C_{04}=-4C_{13}-3C_{22}.
	\end{equation}
	By \eqref{eq:lambdas} and \eqref{eq:tau-conditional},
	$C_{11}=\lambda_1$, $C_{12}=a$, $C_{22}=\lambda_2$,
	$\delta^2=\lambda_2-a^2/\lambda_1$, and
	$\tau^2=\delta^2-\lambda_1^2$. Fourth-order Taylor expansion of $C$ at $(t,t)$ gives
	\begin{align*}
		c_-&=1-\frac{\lambda_1}{2}d^2
		+\frac{a}{2}d^3+\frac{C_{04}}{24}d^4+o(d^4),\\
		c_+&=1-\frac{\lambda_1}{2}d^2
		-\frac{a}{2}d^3+\frac{C_{04}}{24}d^4+o(d^4),\\
		c_0&=1-2\lambda_1d^2
		+\frac{C_{04}+3\lambda_2-4C_{13}}{12}d^4+o(d^4).
	\end{align*}
	For the peak probability, Remark~\ref{remark:gamma} gives
	$p_{t,d}=(2\pi)^{-1}\arccos\gamma_{t,d}$, where $\gamma_{t,d}
	=-q_{t,d}$,
	\begin{align*}
		c_-+c_+-1-c_0
		&=\lambda_1d^2
		+\frac{4C_{13}-3\lambda_2}{12}d^4+o(d^4),\\
		2\sqrt{(1-c_-)(1-c_+)}
		&=\lambda_1d^2
		\left\{1-
		\left(\frac{C_{04}}{12\lambda_1}
		+\frac{a^2}{2\lambda_1^2}\right)d^2
		+o(d^2)\right\}.
	\end{align*}
	Substitution of \eqref{eq:diagonal-derivative-identities} therefore yields
	\begin{equation}\label{eq:smooth-gamma-expansion}
	\gamma_{t,d}
	=1-\frac{1}{2\lambda_1}
	\left(\lambda_2-\frac{a^2}{\lambda_1}\right)d^2
	+o(d^2)
	=1-\frac{\delta^2}{2\lambda_1}d^2+o(d^2).
	\end{equation}
	Since $\arccos(1-y)\sim\sqrt{2y}$,
	we obtain $p_{t,d}/d\to\delta/(2\pi\sqrt{\lambda_1})=J_t$.
	
	For the conditional Gaussian law, the same expansions give
	\begin{align*}
		c_0-c_-c_+
		&=-\lambda_1d^2
		+\left(\frac{\lambda_2}{4}-\frac{C_{13}}{3}
		-\frac{\lambda_1^2}{4}\right)d^4+o(d^4),\\
		\sqrt{(1-c_-^2)(1-c_+^2)}
		&=\lambda_1d^2
		\left\{1-
		\left(\frac{C_{04}}{12\lambda_1}
		+\frac{\lambda_1}{4}
		+\frac{a^2}{2\lambda_1^2}\right)d^2
		+o(d^2)\right\}.
	\end{align*}
	After another use of \eqref{eq:diagonal-derivative-identities} yields
	\begin{equation}\label{eq:smooth-xi-expansion}
		\xi_{t,d}=-1+\frac{\tau^2}{2\lambda_1}d^2+o(d^2),
		\qquad
		\frac{b_{t,d}^{\pm}}d\longrightarrow\frac{\sqrt{\lambda_1}}2.
	\end{equation}
	Because $1-c_\pm^2=\lambda_1d^2+o(d^2)$, the condition in Theorem~\ref{thm:exact} is satisfied.
	
	Apply Lemma~\ref{lem:antipodal-uniform} with
	$A_d=b_{t,d}^{-}x$, $B_d=b_{t,d}^{+}x$, and
	$\beta=\tau^2/(2\lambda_1)$. Multiplication by $\phi(x)$, together with $\tau=\delta\sqrt{1-\eta_t^2}$ and 
	$\lambda_1/\tau=\eta_t/\sqrt{1-\eta_t^2}$, gives
	\begin{align*}
		\frac{g_{t,d}(x)}d
		&\to\frac{\phi(x)}{\sqrt{2\pi}}
		\E\!\left[\left(\sqrt{\lambda_1}x
		-\frac{\tau}{\sqrt{\lambda_1}}U\right)_+\right]\\
		&=\frac{\phi(x)}{\sqrt{2\pi\lambda_1}}
		\left\{\tau\phi\!\left(\frac{\lambda_1x}{\tau}\right)
		+\lambda_1x\Phi\!\left(\frac{\lambda_1x}{\tau}\right)\right\}
		=J_t(x).
	\end{align*}
	Combining the two limits yields the pointwise convergence to the first case of \eqref{eq:h-smooth-rho},
	\[
	h_{t,d}(x)=\frac{g_{t,d}(x)/d}{p_{t,d}/d}
	\to \frac{J_t(x)}{J_t}
	=h_{\eta_t}^{\mathrm{SM}}(x), \quad x\in \R.
	\]
	Since both $h_{t,d}$ and
	$h_{\eta_t}^{\mathrm{SM}}$ are probability densities, Scheff\'e's theorem gives $L^1$ and total variation convergence
	$\int_{\R}|h_{t,d}(x)-h_{\eta_t}^{\mathrm{SM}}(x)|\dd x\to0$.
	
	(ii) We next prove the degenerate case $\eta_t=1$ (i.e., $\tau=0$). The Taylor calculations above do not use $\tau>0$. Here $\delta=\lambda_1$ and $\tau=0$, so
	\eqref{eq:smooth-gamma-expansion} gives $p_{t,d}/d\to\sqrt{\lambda_1}/(2\pi)$.
	
	Apply Lemma~\ref{lem:antipodal-uniform} with
	$A_d=b_{t,d}^{-}x$, $B_d=b_{t,d}^{+}x$, and
	$\beta=0$. We obtain
	\begin{align*}
		\frac{g_{t,d}(x)}d
		\to\frac{\phi(x)}{\sqrt{2\pi}}
		(\sqrt{\lambda_1}x)_+
		=\frac{\sqrt{\lambda_1}}{2\pi}
		x e^{-x^2/2}\1_{\{x>0\}} = J_t(x), \quad x\in \R.
	\end{align*}
	Taking the ratio of two limits yields \eqref{eq:h-smooth-rho}. Scheff\'e's theorem again proves $L^1$ convergence.
\end{proof}

The endpoint includes the process
$X(t)=Z_1\cos(\omega t)+Z_2\sin(\omega t)$, where $Z_1,Z_2$ are
independent standard Gaussian variables and $\omega\ne0$. Its covariance
is $r(s)=\cos(\omega s)$, $\eta=1$ and the covariance matrix of $(X(t-d), X(t), X(t+d))$ has rank two. For $0<|\omega|d<\pi$, $\xi_d=-1$ exactly and
$p_d=|\omega|d/(2\pi)$. The degenerate representation in Lemma \ref{lem:antipodal-uniform} and Theorem \ref{thm:smooth-limit} therefore covers an actual model, rather than only a limiting case of nonsingular matrices.

\begin{corollary}[Limit of expected counts]
	\label{cor:smooth-counts}
	Suppose that $C$ has continuous partial derivatives through total order
	four on a neighborhood of the diagonal
	$\{(t,t):t\in I\}$, where $I=[a,b]$, and that $\inf_{t\in I}\lambda_1(t)>0$.
	Then, for every fixed $u\in\R$,
	\begin{align}\label{eq:smooth-count-limits}
		\E M_d(I)
		\to \int_a^b J_t\dd t,\qquad
		\E M_{u,d}(I)
		\to
		\int_a^b\int_u^\infty J_t(x)\dd x\dd t.
	\end{align}
\end{corollary}

\begin{proof}
	Continuity of the fourth derivatives and compactness make every Taylor
	remainder in the proof of Theorem \ref{thm:smooth-limit} uniform over $t\in\mathcal G_d^\circ$. All derivative
	variances are bounded there, and $\lambda_1(t)$ is bounded away from
	zero. Moreover $\delta_t^2\geq\lambda_1(t)^2$, so
	$\delta_t^2/(2\lambda_1(t))$ is bounded away from zero. The expansion
	\eqref{eq:smooth-gamma-expansion} and the arccosine asymptotic give
	\begin{equation}\label{eq:uniform-lim-1}
	\sup_{t\in\mathcal G_d^\circ}
	\left|\frac{p_{t,d}}d-J_t\right|\to 0.
\end{equation}
	For each fixed $x$, apply Lemma~\ref{lem:antipodal-uniform} with index
	$\theta=t$, $A_{t,d}=b_{t,d}^{-}x$, $B_{t,d}=b_{t,d}^{+}x$, and
	$\beta_t=\tau_t^2/(2\lambda_1(t))$. Uniformity of
	\eqref{eq:smooth-xi-expansion} gives all the lemma's hypotheses. Thus
	\[
	\sup_{t\in\mathcal G_d^\circ}
	\left|\frac{g_{t,d}(x)}d-J_t(x)\right|\to 0.
	\]
	
	The probability bound in Lemma~\ref{lem:antipodal-uniform}, together with uniformly bounded
	$b_{t,d}^{\pm}/d$ and $(1+\xi_{t,d})/d^2$, gives
	$g_{t,d}(x)/d\leq K(1+|x|)\phi(x)$ for some constant $K>0$. The representation
	in the first line of \eqref{eq:J_tx} gives the same bound for $J_t(x)$. Hence
	\begin{equation*}
		\frac{g_{t,d}(x)}d+J_t(x)\leq K(1+|x|)\phi(x)
	\end{equation*}
	uniformly in $t,x$ and small $d$. Therefore, by the dominated convergence theorem,
	\begin{equation}\label{eq:uniform-lim-2}
	\sup_{t\in\mathcal G_d^\circ}
	\left|
	\int_u^\infty\frac{g_{t,d}(x)}d\dd x
	-\int_u^\infty J_t(x)\dd x
	\right|
	\to 0.
	\end{equation}
	
	Finally,
	\[
	\E M_d(I)
	=d\sum_{t\in\mathcal G_d^\circ}\frac{p_{t,d}}d,
	\qquad
	\E M_{u,d}(I)
	=d\sum_{t\in\mathcal G_d^\circ}
	\int_u^\infty \frac{g_{t,d}(x)}d \dd x.
	\]
	The functions $J_t$ and $J_t(x)$ are continuous in $t$, including where
	$\tau_t=0$. The envelope also makes
	$t\mapsto\int_u^\infty J_t(x)\dd x$ continuous. Since
	$d|\mathcal G_d^\circ|$ is bounded, the uniform limits \eqref{eq:uniform-lim-1} and \eqref{eq:uniform-lim-2} reduce both
	expressions to their corresponding Riemann sums. This proves the two
	asserted limits of expected counts in \eqref{eq:smooth-count-limits}. 
\end{proof}

The second limit in \eqref{eq:smooth-count-limits} can also be written as $\E M_{u,d}(I)\to
\int_a^bJ_t\overline H_{\eta_t}^{\mathrm{SM}}(u)\dd t$, where $\overline H_{\eta_t}^{\mathrm{SM}}$ is the tail of the smooth peak height distribution. By \cite{Cheng25} or direct integration of
\eqref{eq:h-smooth-rho}, 
\begin{equation*}
	\overline H_{\eta_t}^{\mathrm{SM}}(u)=
	\begin{cases}
		\Psi\left(\frac{u}{\sqrt{1-\eta_t^2}}\right) + \sqrt{2\pi}\eta_t\phi(u)\Phi\left(\frac{\eta_tu}{\sqrt{1-\eta_t^2}}\right)  ,&\quad 0<\eta_t<1,\\
		e^{-\left(u_+\right)^2/2},&\quad \eta_t=1.
	\end{cases}
\end{equation*}

\subsection{Stationary specialization}

Let $X$ be stationary with covariance $C(s,t)=r(t-s)$. Omit the site subscript and write
\[
\lambda_1=-r''(0)>0,\qquad
\lambda_2=r^{(4)}(0),\qquad
\eta=\frac{\lambda_1}{\sqrt{\lambda_2}}\in(0,1].
\]
Stationarity gives $a=0$, $\delta^2=\lambda_2$, and
$J=(2\pi)^{-1}\sqrt{\lambda_2/\lambda_1}$.
Consequently, Theorem~\ref{thm:smooth-limit} yields
the following specialization.
\begin{corollary}[Stationary smooth limit]
	\label{cor:stationary-smooth}
	Suppose that $r$ is $C^4$ in a neighborhood of the origin and
	$\lambda_1>0$. Then, $p_d/d\to J$ and $g_d(x)/d \to J(x)$ for every $x\in\R$.
	Consequently, $h_d(x)$ converges pointwise and in $L^1(\R)$ to the smooth peak height density
	\begin{equation}\label{eq:h-eta}
			 h_\eta^{\mathrm{SM}}(x)
			= \begin{cases}
			\sqrt{1-\eta^2}\,
			\phi\!\left(\frac{x}{\sqrt{1-\eta^2}}\right)
			+\sqrt{2\pi}\,\eta x\phi(x)
			\Phi\!\left(
			\frac{\eta x}{\sqrt{1-\eta^2}}
			\right),& \quad 0<\eta<1,\\
			x e^{-x^2/2}\1_{\{x>0\}} ,& \quad \eta=1,
		\end{cases}
	\end{equation}
	and thus the discrete peak height distributions converge in total variation.
\end{corollary}

\section{Rough Covariance and Gaussian Peak Heights}\label{sec:rough}

We characterize when discrete peak selection becomes asymptotically
neutral with respect to height. We give a necessary and sufficient
criterion for a standard Gaussian height limit. The criterion applies
to general nonstationary covariance and allows the peak probability
to vanish. An explicit finite-grid bound yields uniform
approximations across sites. Under stationarity,
the criterion reduces to a ratio of covariance increments.

We also give a complete spectral classification for a fixed stationary
process. The process is assumed to be nonconstant and mean-square
continuous. A finite fourth spectral moment yields the smooth height
law of Section~\ref{sec:smooth}, whereas an infinite fourth moment
yields the marginal Gaussian law.

\subsection{A nonstationary criterion for Gaussian limit}

At a fixed interior point $t$, let
\begin{equation*}
	v_{t,d}^{\pm}=1-c_\pm.
\end{equation*}
Assume $-1<c_\pm<1$ for all sufficiently small $d$, and define
\begin{equation*}
	Z_{-,d}=\frac{X(t)-X(t-d)}{\sqrt{2v_{t,d}^{-}}},\qquad
	Z_{+,d}=\frac{X(t)-X(t+d)}{\sqrt{2v_{t,d}^{+}}}.
\end{equation*}
These are standard normal with correlation $q_{t,d}$ in
\eqref{eq:q-general}, and
$\mathcal A_{t,d}=\{Z_{-,d}>0,Z_{+,d}>0\}$.
Moreover, $q_{t,d}>-1$: equality would force $Z_{-,d}+Z_{+,d}=0$
almost surely, whereas its covariance with $X(t)$ is
$\sqrt{v_{t,d}^{-}/2}+\sqrt{v_{t,d}^{+}/2}>0$.
Thus $p_{t,d}>0$ by \eqref{eq:h-exact}.

The standardized comparison sum and its height correlation are
\begin{equation}\label{eq:nonstationary-s-eta}
	S_{t,d}=\frac{Z_{-,d}+Z_{+,d}}{\sqrt{2(1+q_{t,d})}},\qquad
	\eta_{t,d}=\Corr\{X(t),S_{t,d}\}
	=\frac{\sqrt{v_{t,d}^{-}}+\sqrt{v_{t,d}^{+}}}
	{2\sqrt{1+q_{t,d}}}.
\end{equation}
Under the smoothness assumptions in Theorem \ref{thm:smooth-limit},
\[
v_{t,d}^{\pm} = \frac{\lambda_1(t)}{2}d^2 + o(d^2), \quad 1+q_{t,d} = \frac{\delta_t^2}{2\lambda_1(t)}d^2 + o(d^2),
\]
yielding that $\eta_{t,d}^2 \to \lambda_1(t)^2/\delta_t^2 = \eta_t^2 \in (0,1]$. Thus the finite-grid parameter agrees asymptotically with the smooth
parameter. Define
\begin{equation}\label{eq:R}
	R_{t,d}=\frac{v_{t,d}^{-}+v_{t,d}^{+}}{1+q_{t,d}}.
\end{equation}
Then
\begin{equation}\label{eq:nonstationary-eta-bounds}
	\frac{R_{t,d}}4\le\eta_{t,d}^2\le\frac{R_{t,d}}2.
\end{equation}
Thus $R_{t,d}\to0$ is equivalent to loss of correlation between height
and the standardized comparison sum. In the smooth regime,
$R_{t,d}\to2\eta_t^2>0$, so this dependence persists.

\begin{theorem}[Nonstationary criterion for Gaussian limit]
	\label{thm:rough-nonstationary}
	Fix an interior point $t$ and assume $-1<c_\pm<1$ for all sufficiently
	small $d$.  The following are equivalent as $d\to 0$:
	\begin{enumerate}[label=\textup{(\roman*)}]
		\item $R_{t,d}\to0$;
		\item $\eta_{t,d}\to0$;
		\item $\mathcal L\{X(t)\mid\mathcal A_{t,d}\}\Rightarrow N(0,1)$;
		\item $h_{t,d}(x)\to\phi(x)$ for every $x\in\R$;
		\item $\int_\R|h_{t,d}(x)-\phi(x)|\dd x\to0$.
	\end{enumerate}
	Moreover, whenever $R_{t,d}\le1/4$,
	\begin{align}
		\left|\frac{h_{t,d}(x)}{\phi(x)}-1\right|
		&\le\left(\frac\pi2+2\sqrt{\frac\pi3}|x|\right)\sqrt{R_{t,d}},
		\qquad x\in\R,\label{eq:rough-density-error}\\
		\int_\R|h_{t,d}(x)-\phi(x)|\dd x
		&\le C_0\sqrt{R_{t,d}},\qquad
		C_0=\frac\pi2+2\sqrt{\frac23}<4.
		\label{eq:rough-L1-error}
	\end{align}
\end{theorem}

\begin{proof}
	Conditions \textup{(i)} and \textup{(ii)} are equivalent by
	\eqref{eq:nonstationary-eta-bounds}.  We prove the quantitative bounds
	first.  Suppress $t,d$, let $v=v^-+v^+$, and write
	\[
	\rho_\pm=\sqrt{v^\pm/2},\qquad
	D=\sqrt{(1-\rho_-^2)(1-\rho_+^2)},\qquad
	\varepsilon=1+q,\qquad e=1+\xi.
	\]
	The parameters in \eqref{eq:xi-general} satisfy
	\begin{equation*}
		b^\pm=\frac{\rho_\pm}{\sqrt{1-\rho_\pm^2}},\qquad
		\xi=\frac{q-\rho_-\rho_+}{D},\qquad
		e=\frac{\varepsilon-(1-D+\rho_-\rho_+)}{D}.
	\end{equation*}
	Since $\varepsilon=1+q\le2$, the restriction $R=v/\varepsilon\le1/4$ gives $v\le1/2$ and $D\ge3/4$. Also
	$1-D\le \rho_-^2 + \rho_+^2 = v/2$, $\rho_-\rho_+\le v/4$ and $\rho_\pm^2\le v/2\le 1/4$.  Therefore
	\begin{equation}\label{eq:rough-parameter-errors}
		|e-\varepsilon|
		=\frac{|q(1-D)-\rho_-\rho_+|}{D}\le v,
		\qquad (b^-+b^+)^2\le\frac43v.
	\end{equation}
	
	Let $F(s)=\arccos(1-s)/(2\pi)$ for $0\le s\le2$.  The orthant identity
	gives $p_{t,d}=F(\varepsilon)$ and $\Phi_2(0,0;\xi)=F(e)$.
	For $s\in[0,2]$,
	\[
	F(s)\ge\frac{\sqrt{2s}}{2\pi},\qquad
	|\arccos u-\arccos w|\le\frac\pi{\sqrt2}\sqrt{|u-w|},
	\quad u,w\in[-1,1].
	\]
	The first bound follows from $1-\cos y\le y^2/2$.
	For the second, write $A=\arccos u$, $B=\arccos w$, and
	$\delta=|A-B|$; then
	$|u-w|\ge2\sin^2(\delta/2)\ge2\delta^2/\pi^2$.
	Consequently,
	\begin{equation}\label{eq:rough-orthant-error}
		\left|\frac{\Phi_2(0,0;\xi)}{p_{t,d}}-1\right| \le \frac{1}{p_{t,d}}\frac{1}{2\pi}\frac{\pi}{\sqrt{2}}\sqrt{|e-\varepsilon|}
		\le\frac\pi2\sqrt{R}.
	\end{equation}
	The change from zero thresholds to $(b^-x,b^+x)$ is contained in the two
	marginal strips.  Since standard normal marginal densities are bounded
	by $\phi(0)$, 
	\[
	|\Phi_2(b^-x,b^+x;\xi)-\Phi_2(0,0;\xi)|
	\le\phi(0)|x|(b^-+b^+).
	\]
	Dividing by $p_{t,d}\ge\sqrt{\varepsilon}/(\pi\sqrt2)$ and using
	\eqref{eq:rough-parameter-errors} bounds this relative error by
	$2\sqrt{\pi/3}|x|\sqrt R$.
	Together with \eqref{eq:rough-orthant-error} and
	\eqref{eq:h-exact}, this proves \eqref{eq:rough-density-error}.
	Integration against $\phi$, using
	$\int_\R|x|\phi(x)\dd x=\sqrt{2/\pi}$, proves
	\eqref{eq:rough-L1-error}.
	
	Under \textup{(i)}, these bounds prove \textup{(iv)} and \textup{(v)}.
	Also, \textup{(iv)} implies \textup{(v)} by Scheff\'e's theorem, and
	\textup{(v)} implies \textup{(iii)}, with
	\begin{equation}\label{eq:rough-TV}
		d_{\mathrm{TV}}\bigl(\mathcal L\{X(t)\mid\mathcal A_{t,d}\},N(0,1)\bigr)
		=\frac12\int_\R|h_{t,d}(x)-\phi(x)|\dd x.
	\end{equation}
	
	Finally, suppose \textup{(iii)} holds.  Suppress $t,d$, write
	$S=S_{t,d}$ and $\eta=\eta_{t,d}$, and first take $q<1$.
	The variable $T=(Z_{-,d}-Z_{+,d})/\sqrt{2(1-q)}$ is standard normal
	and independent of $S$.  The Gaussian conditioning formula gives
	\[
	X(t)=\eta S+\kappa T+\sigma G,\quad
	\kappa=\Corr\{X(t),T\}=\frac{\rho_- - \rho_+}{\sqrt{2(1-q)}},\quad
	\sigma^2=1-\eta^2-\kappa^2\ge0,
	\]
	with $G$ independent standard normal; the formula allows $\sigma=0$.
	The peak event can be written as
	$\mathcal A_{t,d} = \{S>0,\ |T|<(\tan\vartheta)S\}$, where
	\[
	\vartheta= \arctan\sqrt{\frac{1+q}{1-q}} \in(0,\pi/2).
	\]
	In polar coordinates $S=\mathcal R\cos\alpha$ and
	$T=\mathcal R\sin\alpha$, the peak event conditioning restricts the angle $\alpha$ uniform on
	$(-\vartheta,\vartheta)$ and leaves the independent radius $\mathcal R$
	unchanged.  Since $\mathcal R$ has Rayleigh distribution with $\E\mathcal R=\sqrt{\pi/2}$ and
	$\E\mathcal R^2=2$, symmetry and angular integration give
	\begin{equation}\label{eq:mean-bound}
	\E\{X(t)\mid\mathcal A_{t,d}\} = \eta \, \E\mathcal R \, \E\{\cos\alpha \mid |\alpha|<\vartheta\}
	=\eta\sqrt{\frac\pi2}\frac{\sin\vartheta}{\vartheta}
	\ge\sqrt{\frac2\pi}\,\eta.
    \end{equation}
	Also, Cauchy--Schwarz yields
	\[
	\E\{X(t)^2\mid\mathcal A_{t,d}\}
	\le(\eta^2+\kappa^2)
	\E(S^2+T^2\mid\mathcal A_{t,d})+\sigma^2
	=2(\eta^2+\kappa^2)+\sigma^2\le2.
	\]
	
	If $q=1$, then $Z_{-,d}=Z_{+,d}$ almost surely and  $\mathcal A_{t,d}=\{S>0\}$ with $S=Z_{-,d}=Z_{+,d}$. The formula of Gaussian conditional on $S$ gives
	conditional mean $\sqrt{2/\pi}\,\eta$ and second moment $1$.
	Thus the same bounds hold for all possible $q$.
	The second-moment bound ensures uniform integrability, so \textup{(iii)}
	implies $\E\{X(t)\mid\mathcal A_{t,d}\}\to0$.
	The mean bound in \eqref{eq:mean-bound} then gives $\eta_{t,d}\to0$, proving \textup{(ii)} and
	completing the equivalence.
\end{proof}

\begin{remark}[A unified family of limiting height densities]\label{remark:unify}
	For fixed $t$, recall the smooth peak height density in \eqref{eq:h-smooth-rho}
	\[
	h_{\eta_t}^{\mathrm{SM}}(x)
	=\sqrt{1-\eta^2}\,
	\phi\!\left(\frac{x}{\sqrt{1-\eta_t^2}}\right)
	+\sqrt{2\pi}\,\eta_t x\phi(x)
	\Phi\!\left(\frac{\eta_t x}{\sqrt{1-\eta_t^2}}\right),
	\qquad 0<\eta_t<1.
	\]
	Viewed as a family indexed by $\eta_t$, these densities have the
	pointwise endpoint limits
	\[
	\lim_{\eta_t\downarrow0}h_{\eta_t}^{\mathrm{SM}}(x)=\phi(x),
	\qquad
	\lim_{\eta_t\uparrow1}h_{\eta_t}^{\mathrm{SM}}(x)
	=xe^{-x^2/2}\1_{\{x>0\}},
	\qquad x\in\R.
	\]
	Define $h_0^{\mathrm{SM}}$ and $h_1^{\mathrm{SM}}$ by these limits.
	Since the endpoint limits are probability densities,
	Scheff\'e's theorem also gives convergence in $L^1(\R)$ and hence
	in total variation.
	
	The extended family $\{h_{\eta_t}^{\mathrm{SM}}:0\le\eta_t\le1\}$
	therefore unifies the Gaussian, smooth, and Rayleigh limiting
	height laws established in this paper. In the smooth nonstationary regime, the limiting density is $h_{\eta_t}^{\mathrm{SM}}$, where $\eta_t\in(0,1]$. The endpoint $\eta_t=1$ corresponds to the Rayleigh distribution.
	The Gaussian law is the endpoint extension at $\eta_t=0$.
	The covariance assumptions in the corresponding infill theorems
	determine which member of the family arises.
\end{remark}

\begin{remark}[Uniform convergence of $h_{t,d}$]
	In (iv), the convergence for $h_{t,d}(x)\to\phi(x)$  is uniform on $\R$. This is because, multiplying \eqref{eq:rough-density-error} by $\phi(x)$ gives
	\[
	\sup_{x\in \R}\left|h_{t,d}(x)-\phi(x)\right|
	\le M\sqrt{R_{t,d}}\, \sup_{x\in \R}(1+|x|)\phi(x)\le M'\sqrt{R_{t,d}},
	\]
	where $M$ and $M'$ are positive constants. 
\end{remark}

\begin{corollary}[Gaussian limit with positive peak probability]
	\label{thm:rough-general}
	Suppose $c_\pm\to1$, $1-c_\pm>0$ for all sufficiently small $d$, and
	$q_{t,d}\to q_t\in(-1,1]$.  Then
	\[
	p_{t,d}\to\frac14+\frac{\arcsin q_t}{2\pi}>0,
	\qquad h_{t,d}(x)\to \phi(x),\quad x\in\R,
	\]
	and the density convergence also holds in $L^1$.
\end{corollary}

\begin{proof}
	The assumptions imply $-1<c_\pm<1$, while the numerator of
	$R_{t,d}$ tends to zero and its denominator tends to $1+q_t>0$.
	The results follow directly from Theorem~\ref{thm:rough-nonstationary} and \eqref{eq:h-exact}.
\end{proof}

For example, if locally
\[
c_\pm=1-a_\pm d^\alpha+o(d^\alpha),\qquad
c_0=1-bd^\alpha+o(d^\alpha),\qquad a_\pm>0,
\]
then $q_{t,d}\to(a_-+a_+-b)/(2\sqrt{a_-a_+})$.
Corollary~\ref{thm:rough-general} applies when this limit exceeds $-1$.
If $q_{t,d}\to-1$, Theorem~\ref{thm:rough-nonstationary} still applies
provided $v_{t,d}^{-}+v_{t,d}^{+}=o(1+q_{t,d})$.

\begin{corollary}[Uniform sitewise Gaussian approximations]
	\label{cor:rough-uniform}
	For each sufficiently small $d$, let $\mathcal T_d$ be a nonempty set of sites
	with $t,t\pm d\in I$ satisfying $-1<c_\pm<1$, and suppose
	\[
	R_d^*:=\sup_{t\in \mathcal T_d}R_{t,d}\to 0.
	\]
	Then the density convergence in Theorem~\ref{thm:rough-nonstationary} is
	uniform across $t\in \mathcal T_d$, both in $L^1$ and in absolute error on $\R$.
	For sufficiently small $d$,
	\begin{equation*}
		\begin{aligned}
			\sup_{t\in \mathcal T_d}\int_\R|h_{t,d}(x)-\phi(x)|\dd x
			&\le C_0\sqrt{R_d^*},\\
			\sup_{t\in \mathcal T_d,\,u\in\R}
			|\overline H_{t,d}(u)-\Psi(u)|&\le\frac{C_0}{2}\sqrt{R_d^*}.
		\end{aligned}
	\end{equation*}
\end{corollary}

\begin{proof}
	Eventually $R_d^*\le1/4$, so apply
	\eqref{eq:rough-density-error}--\eqref{eq:rough-L1-error} uniformly in
	$t$.  The survival error is bounded by the total variation distance
	in \eqref{eq:rough-TV}.  
\end{proof}

\begin{remark}[Relative tails in a controlled threshold range]
	\label{rem:rough-relative-tails}
	For $u\ge0$ and $R_{t,d}\le1/4$, integrating
	\eqref{eq:rough-density-error} over $(u,\infty)$ gives
	\begin{equation}\label{eq:rough-relative-tail-bound}
		\left|\frac{\overline H_{t,d}(u)}{\Psi(u)}-1\right|
		\le\sqrt{R_{t,d}}\left\{\frac\pi2+
		2\sqrt{\frac\pi3}\frac{\phi(u)}{\Psi(u)}\right\}.
	\end{equation}
	Here $\int_u^\infty x\phi(x)\dd x=\phi(u)$.
	Since $\phi(u)/\Psi(u)=O(1+u)$, if $R_{t,d}\to0$ and
	$u_d\ge0$ satisfies $u_d\sqrt{R_{t,d}}\to0$, then
	$\overline H_{t,d}(u_d)/\Psi(u_d)\to1$.
	This sufficient range comes from the density bound; total variation
	convergence alone supplies only absolute tail accuracy.
\end{remark}

\subsection{The stationary criterion and height--curvature correlation}

For a stationary, unit-variance process, let
\begin{equation*}
	V(d)=1-r(d),\qquad \Delta(d)=4V(d)-V(2d).
\end{equation*}
For the central second difference $D_d=X(d)-2X(0)+X(-d)$,
\[
\Var(D_d)=2\Delta(d),\qquad \Cov\{X(0),D_d\}=-2V(d).
\]
Thus $\Delta(d)$ is one half of the second-difference variance, and
Cauchy--Schwarz gives $\Delta(d)\ge2V(d)^2$.  When $V(d)>0$, define
\begin{equation*}
	\eta_d=-\Corr\{X(0),D_d\}
	=\frac{\sqrt2\,V(d)}{\sqrt{\Delta(d)}}\in(0,1].
\end{equation*}
This is the stationary specialization of \eqref{eq:nonstationary-s-eta} and \eqref{eq:R}:
\begin{equation}\label{eq:stationary-nonstationary-ratios}
	\eta_{t,d}=\eta_d,\qquad
	R_{t,d}=\frac{4V(d)^2}{\Delta(d)}=2\eta_d^2.
\end{equation}

\begin{corollary}[Stationary Gaussian-limit characterization]
	\label{thm:rough-stationary}
	Suppose $-1<r(d)<1$ for all sufficiently small $d$.  Then $\Delta(d)>0$,
	and the condition
	\begin{equation}\label{eq:rough-criterion}
		\frac{V(d)^2}{\Delta(d)}\to 0
	\end{equation}
	is equivalent to weak convergence of the discrete peak height distributions to $N(0,1)$, to pointwise convergence $h_d(x)\to\phi(x)$, and to
	$L^1$ convergence of these densities. The quantitative bounds of
	Theorem~\ref{thm:rough-nonstationary} apply with
	$R_d=4V(d)^2/\Delta(d)$.
\end{corollary}

\begin{proof}
	The results are direct consequences of Theorem~\ref{thm:rough-nonstationary} and
	\eqref{eq:stationary-nonstationary-ratios}.
\end{proof}

The criterion in \eqref{eq:rough-criterion} is equivalent to $\eta_d\to0$, showing that discrete curvature becomes
asymptotically uncorrelated with the height. 

\subsection{The fourth spectral moment boundary}

Assume that $r$ is continuous at the origin, equivalently that $X$ is
mean-square continuous.  By Bochner's theorem, its symmetric spectral
probability measure $F$ satisfies
\[
r(s)=\int_\R\cos(s\lambda)F(\dd\lambda).
\]
Letting $Y_d(\lambda)=1-\cos(d\lambda)$, we have
\begin{equation}\label{eq:Delta}
	V(d)=\int_\R Y_d\dd F,\qquad
	\Delta(d)=2\int_\R Y_d^2\dd F,\qquad
	\eta_d^2=\frac{\{\int_\R Y_d\dd F\}^2}{\int_\R Y_d^2\dd F}.
\end{equation}
Denote by $m_j=\int_\R \lambda^jF(\dd\lambda)$ the $j$-th spectral moment.

\begin{theorem}[Fourth spectral moment classification]\label{thm:spectral-classification}
	Assume that $X$ is mean-square continuous and stationary, and is nonconstant, equivalently $F\ne\delta_0$ or $r\not\equiv1$.
	Then $V(d)$ and $\Delta(d)$ are positive for all sufficiently small $d$,
	and exactly one of the following alternatives holds:
	\begin{enumerate}
		\item If $m_4<\infty$, then $0<m_2<\infty$, $r\in C^4$, and
		\[
		\eta_d\to\eta=\frac{m_2}{\sqrt{m_4}}\in(0,1],\qquad
		\frac{p_d}{d}\to\frac{1}{2\pi}\sqrt{\frac{m_4}{m_2}},\qquad
		h_d(x)\to h_\eta^{\mathrm{SM}}(x),\quad x\in\R.
		\]
		\item If $m_4=\infty$, then $\eta_d\to0$ and $h_d(x)\to\phi(x)$ for
		every $x\in\R$.
	\end{enumerate}
	In both cases the density convergence also holds in $L^1$, and
	hence in total variation.
\end{theorem}

\begin{proof}
	First, $V(d)\to0$ by dominated convergence.  Nonconstancy implies that
	$\Delta(d)>0$ for all sufficiently small $d$.  Otherwise, along a sequence $d_n\downarrow0$,
	the nonnegative integrand in \eqref{eq:Delta} would give
	$d_n\lambda\in2\pi\Z$ for every $n$, for $F$-almost every $\lambda$.
	Since $d_n\lambda\to0$, this forces $\lambda=0$, contradicting
	nonconstancy.  Moreover, $\Delta(d)>0$ implies $V(d)>0$, for all sufficiently small $d$, through the same spectral identities.
	
	Suppose first that $m_4<\infty$.  Since
	\[
	0\le\frac{Y_d(\lambda)}{d^2}\le\frac{\lambda^2}{2},\qquad
	0\le\frac{2Y_d(\lambda)^2}{d^4}\le\frac{\lambda^4}{2},
	\]
	dominated convergence gives
	\[
	\frac{V(d)}{d^2}\to\frac{m_2}{2},\qquad
	\frac{\Delta(d)}{d^4}\to\frac{m_4}{2},\qquad
	\eta_d=\frac{\sqrt2\,V(d)}{\sqrt{\Delta(d)}}
	\to\frac{m_2}{\sqrt{m_4}}=\eta.
	\]
	Nonconstancy gives $m_2>0$, and $m_2^2\le m_4$ by Cauchy--Schwarz.
	Differentiation under the spectral integral gives $r\in C^4$, with
	$-r''(0)=m_2$ and $r^{(4)}(0)=m_4$.  The asserted probability and density limits now follow from
	Corollary~\ref{cor:stationary-smooth}, including the degenerate endpoint
	$m_4=m_2^2$.
	
	Now suppose that $m_4=\infty$.  We first prove that $\eta_d\to0$.
	Otherwise, there would exist $\varepsilon>0$ and a sequence
	$d_n\downarrow0$ such that $\eta_{d_n}\geq\varepsilon$.  Define $W_n(\lambda)=Y_{d_n}(\lambda)/V(d_n)$.
	Then
	\[
	\int_{\R}W_n\,\dd F=1,\qquad
	\int_{\R}W_n^2\,\dd F=\eta_{d_n}^{-2}
	\leq\varepsilon^{-2},
	\]
	so $\{W_n\}$ is uniformly integrable.
	
	If $m_2=\infty$, Fatou's lemma gives
	\[
	\liminf_{n\to\infty}\frac{V(d_n)}{d_n^2}
	\geq\frac12\int_{\R}\lambda^2F(\dd\lambda)=\infty.
	\]
	Since
	$Y_{d_n}(\lambda)/d_n^2\to\lambda^2/2$ for every fixed $\lambda$, it
	follows that $W_n(\lambda)\to0$ for almost every $\lambda$.
	Uniform integrability would then imply
	$\int W_n\,\dd F\to0$, contradicting $\int W_n\,\dd F=1$.  Thus
	$m_2<\infty$, and dominated convergence gives $V(d_n)/d_n^2 \to m_2/2$.
	But
	\[
	\int_{\R}\left\{\frac{Y_{d_n}(\lambda)}{d_n^2}\right\}^2F(\dd\lambda)
	=\frac{\{V(d_n)/d_n^2\}^2}{\eta_{d_n}^2}
	\]
	is bounded.  Fatou's lemma would therefore give
	\[
	\frac{m_4}{4}
	\leq\liminf_{n\to\infty}
	\int_{\R}\left\{\frac{Y_{d_n}(\lambda)}{d_n^2}\right\}^2
	F(\dd\lambda)<\infty,
	\]
	contradicting $m_4=\infty$.  Hence $\eta_d\to0$, or equivalently
	$V(d)^2/\Delta(d)\to0$.
	Corollary~\ref{thm:rough-stationary} now gives the Gaussian limit and its
	$L^1$ convergence.
\end{proof}

\begin{remark}\label{rem:fixed-covariance-classification}
	Under the assumptions of Theorem~\ref{thm:spectral-classification},
	a positive limiting value of $\eta_d$ requires $m_4<\infty$.
	The covariance then belongs to $C^4$, and the limiting parameter is
	$\eta=m_2/\sqrt{m_4}\in(0,1]$. If $m_4=\infty$, then $\eta_d\to0$
	and the limiting height law is Gaussian. Finiteness of the fourth
	spectral moment therefore exactly separates the two regimes for
	each fixed nonconstant, mean-square continuous stationary process.
	
	Remark~\ref{remark:unify} connects these limits within the same
	continuous family with the Gaussian
	density at $\eta=0$.
	This interpolation describes variation across covariance models.
	At $\eta=1$, equality $m_4=m_2^2$ forces spectral support on
	$\{\pm\sqrt{m_2}\}$, so $r(s)=\cos(\sqrt{m_2}s)$.
	The conditional curvature variance vanishes, while the limiting
	height law is the nondegenerate Rayleigh law.
\end{remark}

\begin{corollary}[Divergence of expected peak counts in the rough regime]\label{cor:rough-counts}
	Under the assumptions of Theorem~\ref{thm:spectral-classification}, if
	$m_4=\infty$, then $p_d/d\to\infty$.  On an interval $I$ of length $T>0$,
	for every fixed $u\in\R$,
	\[
	\E M_d(I)\to\infty,\qquad \E M_{u,d}(I)\to\infty,\qquad
	\frac{\E M_{u,d}(I)}{\E M_d(I)}\to\Psi(u).
	\]
\end{corollary}

\begin{proof}
	Let $\epsilon_d=\Delta(d)/(2V(d))\in(0,2]$.  The orthant formula and
	$\cos z\ge1-z^2/2$ give
	\[
	\frac{p_d}{d}
	=\frac{\arccos(1-\epsilon_d)}{2\pi d}
	\ge\frac{\sqrt{2\epsilon_d}}{2\pi d}.
	\]
	If $m_2=\infty$, Fatou's lemma gives $V(d)/d^2\to\infty$, and
	$\epsilon_d\ge V(d)$ implies $\epsilon_d/d^2\to\infty$.
	If $0<m_2<\infty$, then $V(d)/d^2\to m_2/2$, whereas Fatou's lemma
	applied to the second identity in \eqref{eq:Delta} gives
	$\Delta(d)/d^4\to\infty$.  Again $\epsilon_d/d^2\to\infty$.
	Thus $p_d/d\to\infty$ in either case.
	With $N_d=|\mathcal G_d^\circ|=T/d+O(1)$, stationarity yields
	$\E M_d(I)=N_dp_d$ and
	\[
	\frac{\E M_{u,d}(I)}{\E M_d(I)}
	=\int_u^\infty h_d(x)\dd x\to\Psi(u)
	\]
	by total variation convergence.  Since $\Psi(u)>0$ for finite $u$,
	both expected counts diverge.
\end{proof}

\subsection{Regular variation and subquartic covariance expansions}

A positive measurable function $L$ is slowly varying at zero if
$L(sd)/L(d)\to1$ as $d\downarrow0$ for every $s>0$.

\begin{corollary}[Nondifferentiable regular variation regime]\label{cor:regular-variation}
	Suppose
	\begin{equation}\label{eq:regular-variation}
		V(d)=1-r(d)\sim cd^\alpha L(d),\qquad c>0,\quad0<\alpha<2,
	\end{equation}
	where $L$ is slowly varying at zero.  Then \eqref{eq:rough-criterion}
	holds, $h_d(x)\to\phi(x)$ for all $x$, and the density convergence is in
	$L^1$.  Moreover,
	\[
	q_d\to q=1-2^{\alpha-1},\qquad
	p_d\to p=\frac14+\frac{1}{2\pi}\arcsin(q)>0.
	\]
	On an interval of length $T$, for every fixed finite threshold $u$,
	\begin{equation*}
		\E M_d\sim\frac{T}{d}p,\qquad
		\E M_{u,d}\sim\frac{T}{d}p\Psi(u).
	\end{equation*}
\end{corollary}

\begin{proof}
	Regular variation gives $V(2d)/V(d)\to2^\alpha$ and $V(d)\to0$.
	Therefore
	\[
	\frac{\Delta(d)}{V(d)}=4-\frac{V(2d)}{V(d)}\to4-2^\alpha>0,\qquad
	\frac{V(d)^2}{\Delta(d)}=\frac{V(d)}{\Delta(d)/V(d)}\to0.
	\]
	Corollary~\ref{thm:rough-stationary} yields the density assertions.
	Since $q_d=1-V(2d)/(2V(d))$, we obtain $q_d\to1-2^{\alpha-1}\in(-1,1/2)$,
	and the orthant identity gives the limit of $p_d$.
	With $N_d=T/d+O(1)$, the exact count identities and total variation
	convergence give
	\[
	\E M_d=N_dp_d\sim\frac{T}{d}p,
	\qquad \overline H_d(u)=\int_u^\infty h_d(x)\dd x\to\Psi(u),
	\]
	which proves the assertion for $\E M_{u,d}$.
\end{proof}

The local form \eqref{eq:regular-variation} is standard in the extreme value
theory of nonsmooth Gaussian processes
\citep{Pickands69,LeadbetterLindgrenRootzen83,Piterbarg96}.

\begin{corollary}[Once but not twice mean-square differentiable]\label{cor:subquartic}
	Suppose, for some $2<\beta<4$,
	\begin{equation*}
		r(d)=1-\frac{\lambda_1}{2}d^2+c|d|^\beta+o(|d|^\beta),
		\qquad\lambda_1,c>0.
	\end{equation*}
	Then \eqref{eq:rough-criterion} holds and $h_d\to\phi$ pointwise and in
	$L^1$.  With $K_\beta=c(2^\beta-4)/\lambda_1$,
	\[
	p_d\sim\frac{\sqrt{2K_\beta}}{2\pi}d^{(\beta-2)/2},\qquad
	\E M_d\sim\frac{T\sqrt{2K_\beta}}{2\pi}d^{(\beta-4)/2}
	\]
	on an interval of length $T$.  For every fixed finite $u$, the latter
	expected count multiplied by $\Psi(u)$ is asymptotic to $\E M_{u,d}$.
\end{corollary}

\begin{proof}
	The covariance expansion gives
	\[
	V(d)=\frac{\lambda_1}{2}d^2-cd^\beta+o(d^\beta),\qquad
	\Delta(d)=c(2^\beta-4)d^\beta+o(d^\beta).
	\]
	Thus
	\[
	\frac{V(d)^2}{\Delta(d)}
	\sim\frac{\lambda_1^2}{4c(2^\beta-4)}d^{4-\beta}\to0,
	\qquad
	\epsilon_d=1+q_d=\frac{\Delta(d)}{2V(d)}\sim K_\beta d^{\beta-2}.
	\]
	Corollary~\ref{thm:rough-stationary} proves the Gaussian limit, while
	$p_d=\arccos(1-\epsilon_d)/(2\pi)\sim\sqrt{2\epsilon_d}/(2\pi)$
	proves the peak probability asymptotic.  Multiplication by
	$N_d=T/d+O(1)$ proves the count assertions.
	Finally, Fatou's lemma and $V(d)/d^2\to\lambda_1/2$ imply $m_2<\infty$;
	dominated convergence then identifies $m_2=\lambda_1$.  The relation
	$\Delta(d)/d^4\to\infty$ excludes $m_4<\infty$.  Hence the process has
	one, but not two, mean-square derivatives.
\end{proof}

\section{Examples and Numerical Comparisons}\label{sec:examples}

\subsection{Ornstein--Uhlenbeck and powered-exponential correlations}

For the Ornstein--Uhlenbeck correlation
$r(s)=e^{-\lambda|s|}$, $\lambda>0$, the identity $r(2d)=r(d)^2$
gives $\xi_d=0$. Thus
\begin{equation}\label{eq:OU-density}
	h_d(x)=\frac{\phi(x)\Phi(b_dx)^2}{p_d},\qquad
	p_d=\frac14+\frac{1}{2\pi}\arcsin\!\left(\frac{1-r(d)}2\right).
\end{equation}
Since $1-r(d)\sim\lambda d$, $p_d\to1/4$ and $h_d\to\phi$
pointwise and in $L^1$. The expected number of interior peaks on an
interval of length $T$ is asymptotic to $T/(4d)$.

More generally, the powered-exponential family $r(s)=\exp(-c|s|^\alpha)$, with $c>0$ and
$0<\alpha<2$, satisfies Corollary~\ref{cor:regular-variation}.
Its limiting peak probability is
$(2\pi)^{-1}\arccos(2^{\alpha-1}-1)$ and its height limit is Gaussian.
At $\alpha=2$ the covariance is smooth and the next example applies.

\subsection{Squared-exponential correlation}

For $r(s)=\exp\{-s^2/(2\ell^2)\}$, one has
$\lambda_1=\ell^{-2}$, $\lambda_2=3\ell^{-4}$, and $\eta=1/\sqrt3$.
The limit is
\begin{equation*}
	h_\eta^{\mathrm{SM}}(x)=\sqrt{\frac23}\,
	\phi\!\left(\sqrt{\frac32}x\right)
	+\sqrt{\frac{2\pi}3}\,x\phi(x)\Phi\!\left(\frac{x}{\sqrt2}\right).
\end{equation*}
Here $J=\sqrt3/(2\pi\ell)$. The limiting height law does not depend
on the correlation length $\ell$, whereas the peak intensity does.

\subsection{The Mat\'ern smoothness phase diagram}

Consider the Mat\'ern correlation function
\begin{equation}\label{eq:matern-covariance}
	r_\nu(s)=\frac{2^{1-\nu}}{\Gamma(\nu)}
	(\theta|s|)^\nu K_\nu(\theta|s|),\qquad \nu,\theta>0,
\end{equation}
where $K_\nu$ is the modified Bessel function.  The parameter $\nu$
controls mean-square differentiability \citep{Stein99,GuttorpGneiting06}. Its one-dimensional spectral density is
\begin{equation}\label{eq:matern-spectrum}
	f_\nu(\lambda)=\frac{\Gamma(\nu+1/2)}{\sqrt\pi\,\Gamma(\nu)}
	\frac{\theta^{2\nu}}{(\theta^2+\lambda^2)^{\nu+1/2}}.
\end{equation}
For $j\geq0$, $m_j<\infty$ precisely when $j<2\nu$. In particular,
Theorem~\ref{thm:spectral-classification} gives a Gaussian peak height
limit for $\nu\leq2$ and a smooth limit for $\nu>2$.
Write $L_d=\log\{1/(\theta d)\}$, which tends to infinity
logarithmically as $d\downarrow0$ for fixed $\theta$.
The covariance expansions follow from the small-argument series for
$K_\nu$; the logarithmic cases use \cite[Section~10.31.1]{DLMF}.

For noninteger $\nu$, the two series branches have leading terms
\[
r_\nu(d)=1+\frac{(\theta d)^2}{4(1-\nu)}+\cdots
+\frac{\Gamma(-\nu)}{2^{2\nu}\Gamma(\nu)}(\theta d)^{2\nu}
\{1+O(d^2)\}.
\]
When $0<\nu<1$, this gives
$1-r_\nu(d)\sim c_\nu d^{2\nu}$, where
$c_\nu=-\Gamma(-\nu)\theta^{2\nu}/\{2^{2\nu}\Gamma(\nu)\}>0$.
When $1<\nu<2$, it gives
\[
r_\nu(d)=1-\frac{\theta^2}{4(\nu-1)}d^2+c_\nu d^{2\nu}
+o(d^{2\nu}),\qquad
c_\nu=\frac{\Gamma(-\nu)\theta^{2\nu}}{2^{2\nu}\Gamma(\nu)}>0.
\]
Corollaries~\ref{cor:regular-variation} and \ref{cor:subquartic}
then give the respective probabilities and count rates.

\begin{table}[tbp]
	\centering\small
	\caption{One-dimensional Mat\'ern infill limits on a fixed interval of
		positive length. Here $L_d=\log\{1/(\theta d)\}$, and $a,c$ denote
		positive constants that may differ between rows. The notation
		$f\asymp g$ means that $f/g$ is bounded above and below by positive
		constants for small $d$.}
	\label{tab:matern}
	\setlength{\tabcolsep}{4pt}
	\begin{tabular}{@{}p{.10\textwidth}p{.33\textwidth}p{.17\textwidth}p{.12\textwidth}p{.16\textwidth}@{}}
		\toprule
		$\nu$ & Local covariance behavior & Peak probability & Height limit & Order of $\E M_d$\\
		\midrule
		$0<\nu<1$ & $1-r_\nu(d)\sim cd^{2\nu}$ & $p_d\to p>0$ & $\phi$ & $d^{-1}$\\[1mm]
		$\nu=1$ & $1-r_1(d)\asymp d^2L_d$ & $p_d\asymp L_d^{-1/2}$ & $\phi$ & $d^{-1}L_d^{-1/2}$\\[1mm]
		$1<\nu<2$ & $r_\nu(d)=1-ad^2+cd^{2\nu}+o(d^{2\nu})$ & $p_d\asymp d^{\nu-1}$ & $\phi$ & $d^{\nu-2}$\\[1mm]
		$\nu=2$ & $r_2(d)=1-ad^2+cd^4L_d+O(d^4)$ & $p_d\asymp dL_d^{1/2}$ & $\phi$ & $L_d^{1/2}$\\[1mm]
		$\nu>2$ & $r_\nu\in C^4$ & $p_d\asymp d$ & $h^{\mathrm{SM}}$ & $1$\\
		\bottomrule
	\end{tabular}
\end{table}

At $\nu=1$, the integer-order expansion gives
\begin{equation*}
	V(d)=\frac{\theta^2}{2}d^2L_d+B_\theta d^2+O(d^4L_d),\qquad
	B_\theta=\frac{\theta^2}{2}\left(\log2-\gamma_{\!E}+\frac12\right),
\end{equation*}
where $\gamma_{\!E}$ is Euler's constant. Since $L_{2d}=L_d-\log2$,
the quadratic constants cancel in $4V(d)-V(2d)$, giving
\[
\Delta(d)\sim2\theta^2(\log2)d^2,\qquad
\frac{V(d)^2}{\Delta(d)}
\sim\frac{\theta^2d^2L_d^2}{8\log2}\to 0.
\]
Consequently,
\begin{equation*}
	p_d\sim\frac{\sqrt{\log2}}{\pi\sqrt{L_d}},\qquad
	\E M_d\sim\frac{T\sqrt{\log2}}{\pi d\sqrt{L_d}}.
\end{equation*}

At $\nu=2$, the first nonsmooth term has quartic logarithmic order:
\begin{equation*}
	r_2(d)=1-\frac{\theta^2}{4}d^2
	+\frac{\theta^4}{16}d^4L_d+O(d^4).
\end{equation*}
Therefore,
\[
\Delta(d)\sim\frac{3\theta^4}{4}d^4L_d,\qquad
\frac{V(d)^2}{\Delta(d)}\sim\frac{1}{12L_d},\qquad
\eta_d^2\sim\frac{1}{6L_d}.
\]
The Gaussian limit persists at this boundary, but the height-curvature
correlation decays only logarithmically. Moreover,
\begin{equation*}
	p_d\sim\frac{\sqrt3\,\theta}{2\pi}d\sqrt{L_d},\qquad
	\E M_d\sim\frac{T\sqrt3\,\theta}{2\pi}\sqrt{L_d}.
\end{equation*}
For both integer boundaries, the expected count above any fixed finite
threshold is asymptotic to the displayed total count multiplied by $\Psi(u)$.

For $\nu>2$, integrating \eqref{eq:matern-spectrum} gives
\begin{equation*}
	m_2=\frac{\theta^2}{2(\nu-1)},\quad
	m_4=\frac{3\theta^4}{4(\nu-1)(\nu-2)},\quad
	\eta=\sqrt{\frac{\nu-2}{3(\nu-1)}},\quad
	J=\frac{\theta}{2\pi}\sqrt{\frac{3}{2(\nu-2)}}.
\end{equation*}
These moments also make explicit why the limiting height shape is
independent of $\theta$, while the count intensity scales with $\theta$. 

Table~\ref{tab:matern} summarizes the peak probability asymptotics, limiting height distributions, and expected peak count rates for the Mat\'ern family over the full range $\nu>0$. 

\subsection{Nonstationary unit-variance examples}\label{sec:nonstationary-example}

Let $Y$ be a centered stationary unit-variance process with smooth
covariance $r$ and $0<m_2,m_4<\infty$. For a $C^4$ function $f$ with
$f'(t)>0$, define $X(t)=Y(f(t))$. Then
\[
C(s,t)=r(f(t)-f(s)),\qquad
c_\pm=r(f(t\pm d)-f(t)),\qquad
c_0=r(f(t+d)-f(t-d)).
\]
These three correlations give the exact nonstationary law in
Theorem~\ref{thm:exact}; the two center-neighbor correlations generally
differ even though the observation grid is uniform.
The mean-square chain rule yields
\[
\lambda_1(t)=f'(t)^2m_2,\qquad
a(t)=f'(t)f''(t)m_2,\qquad
\lambda_2(t)=f''(t)^2m_2+f'(t)^4m_4.
\]
Thus
\begin{equation}\label{eq:time-warp-intensity}
	\delta_t^2=f'(t)^4m_4,\qquad
	\eta_t=\frac{m_2}{\sqrt{m_4}},\qquad
	J_t=\frac{f'(t)}{2\pi}\sqrt{\frac{m_4}{m_2}}.
\end{equation}
The conditional curvature correction removes the $f''(t)^2m_2$
term. 
On $[a,b]$, the limiting expected peak count is
$(f(b)-f(a))\sqrt{m_4/m_2}/(2\pi)$, consistent with the transformed
interval length.

A smooth mixture can also change the limiting height shape. Let
$Y_1,Y_2$ be independent centered stationary unit-variance processes
with smooth covariances $r_j$ and moments $m_{2,j},m_{4,j}$.
For $C^4$ weights $w_1(t)=\cos\theta(t)$ and
$w_2(t)=\sin\theta(t)$, set
\[
X(t)=w_1(t)Y_1(t)+w_2(t)Y_2(t),\qquad
C(s,t)=\sum_{j=1}^2w_j(s)w_j(t)r_j(t-s).
\]
The variance is one because $w_1^2+w_2^2=1$, and direct differentiation gives
\begin{align}
	\lambda_1(t)&=\sum_{j=1}^2\{w_j'(t)^2+m_{2,j}w_j(t)^2\},
	\label{eq:mixture-lambda1}\\
	\lambda_2(t)&=\sum_{j=1}^2\{w_j''(t)^2+4m_{2,j}w_j'(t)^2
	+m_{4,j}w_j(t)^2-2m_{2,j}w_j(t)w_j''(t)\}.
	\label{eq:mixture-lambda2}
\end{align}
Together with $a(t)=\lambda_1'(t)/2$, these expressions determine
$\delta_t$, $\eta_t$, and $J_t$.
For a concrete example, take squared-exponential $r_j$ with length
scales $\ell_1=1$, $\ell_2=2$, and $\theta(t)=t$.
At $t=0$ the values are $\lambda_1=2$, $a=0$, and $\lambda_2=7$,
so $\eta_0=2/\sqrt7$. At $t=\pi/2$ they are
$\lambda_1=5/4$, $a=0$, and $\lambda_2=91/16$, so
$\eta_{\pi/2}=5/\sqrt{91}$. Thus the limiting peak height
densities differ at the two locations. Finite-grid probabilities and
densities follow from the same exact formula.

\begin{figure}[t]
	\begin{center}
		\begin{tabular}{cc}
			\includegraphics[trim=20 30 30 30,clip,width=2.8in]{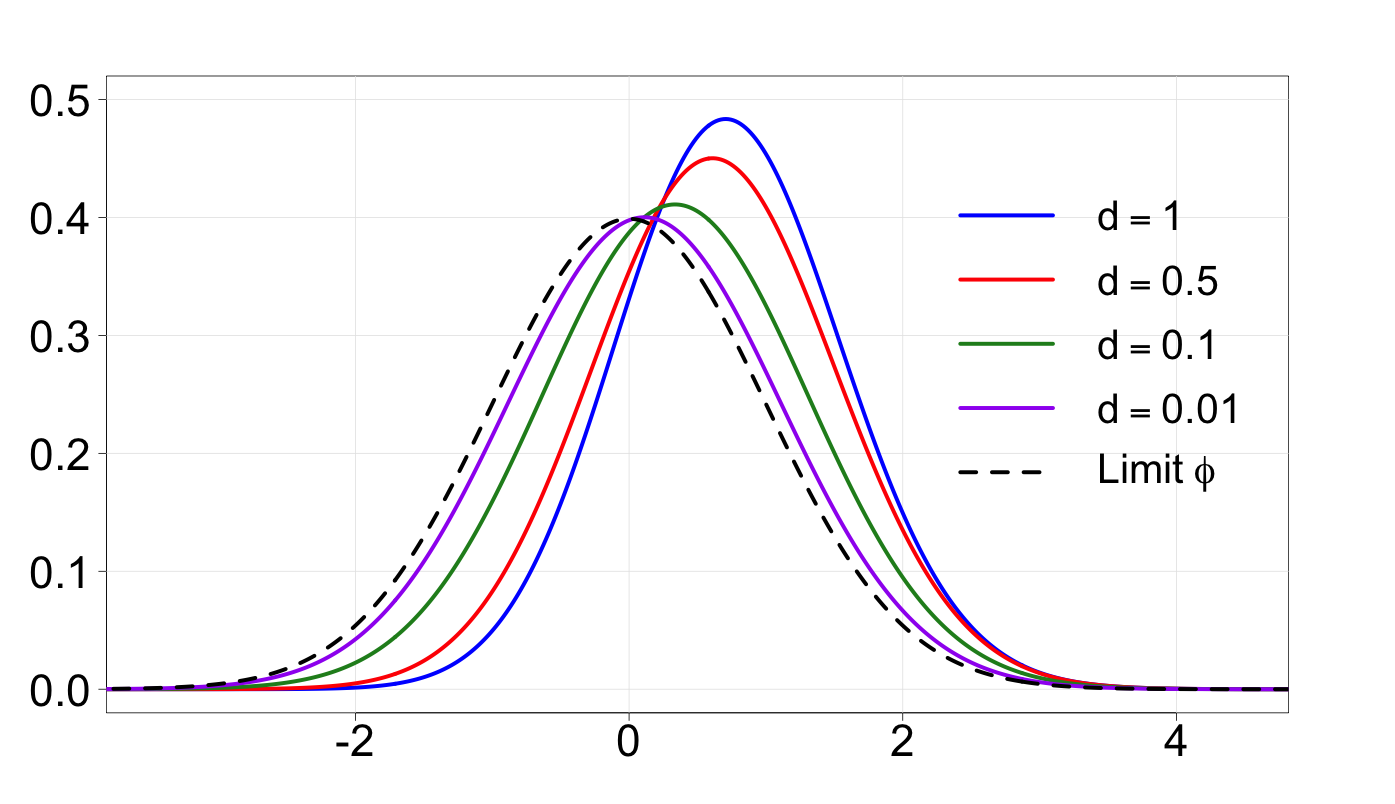} &
			\includegraphics[trim=20 30 30 30,clip,width=2.8in]
			{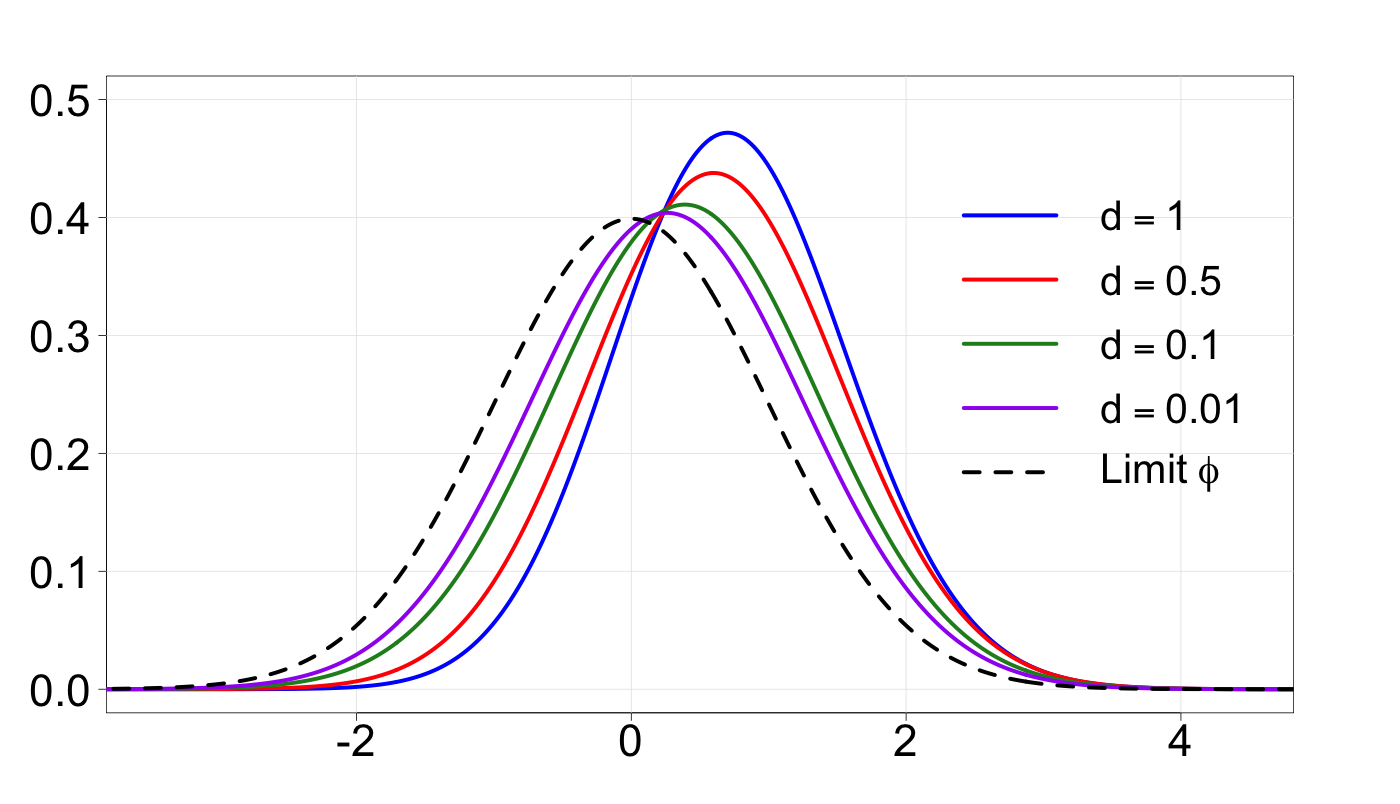}\\
			{\small (a) Ornstein-Uhlenbeck} &
			{\small (b) Mat\'ern $\nu=2$ }
		\end{tabular}
		
		\vspace{3mm}
		
		\begin{tabular}{cc}
			\includegraphics[trim=20 30 30 30,clip,width=2.8in]{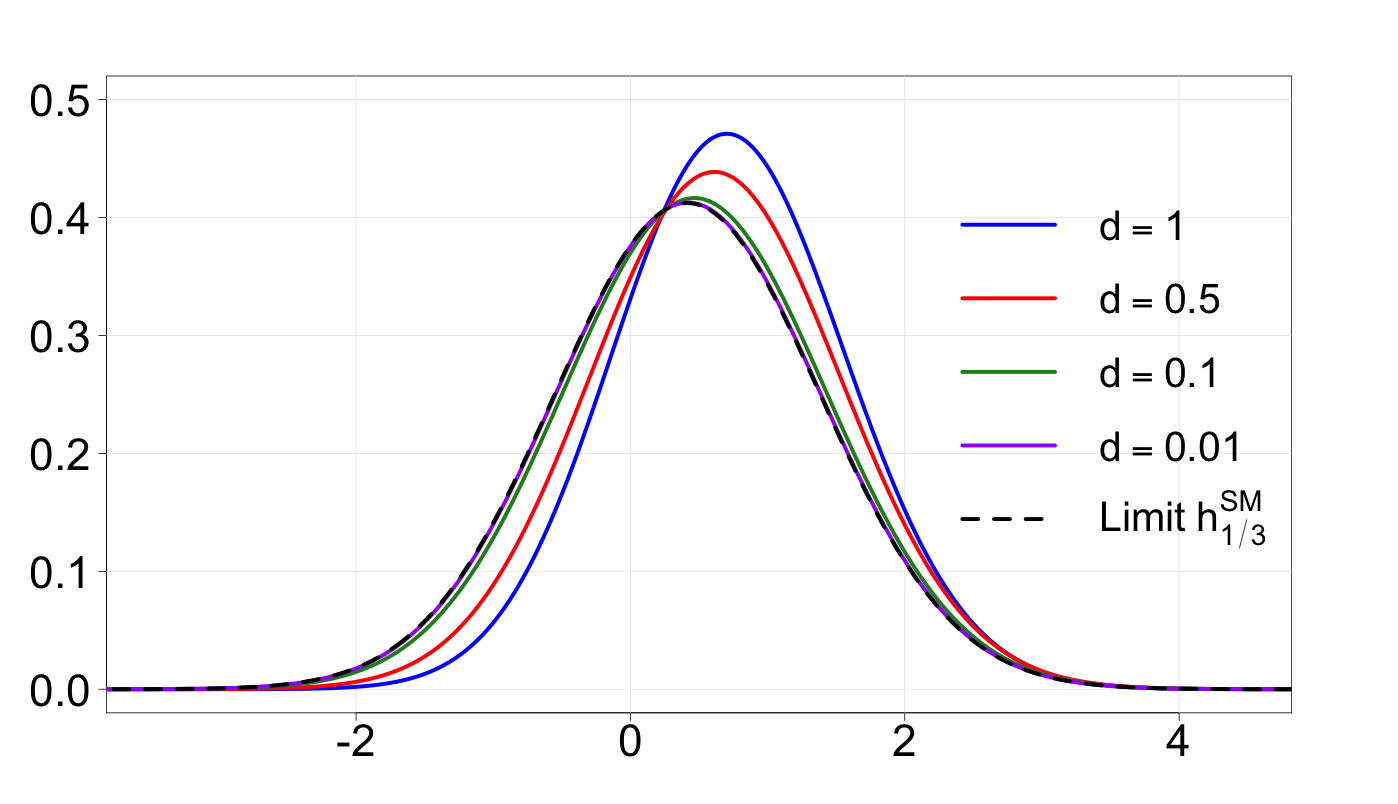} &
			\includegraphics[trim=20 30 30 30,clip,width=2.8in]
			{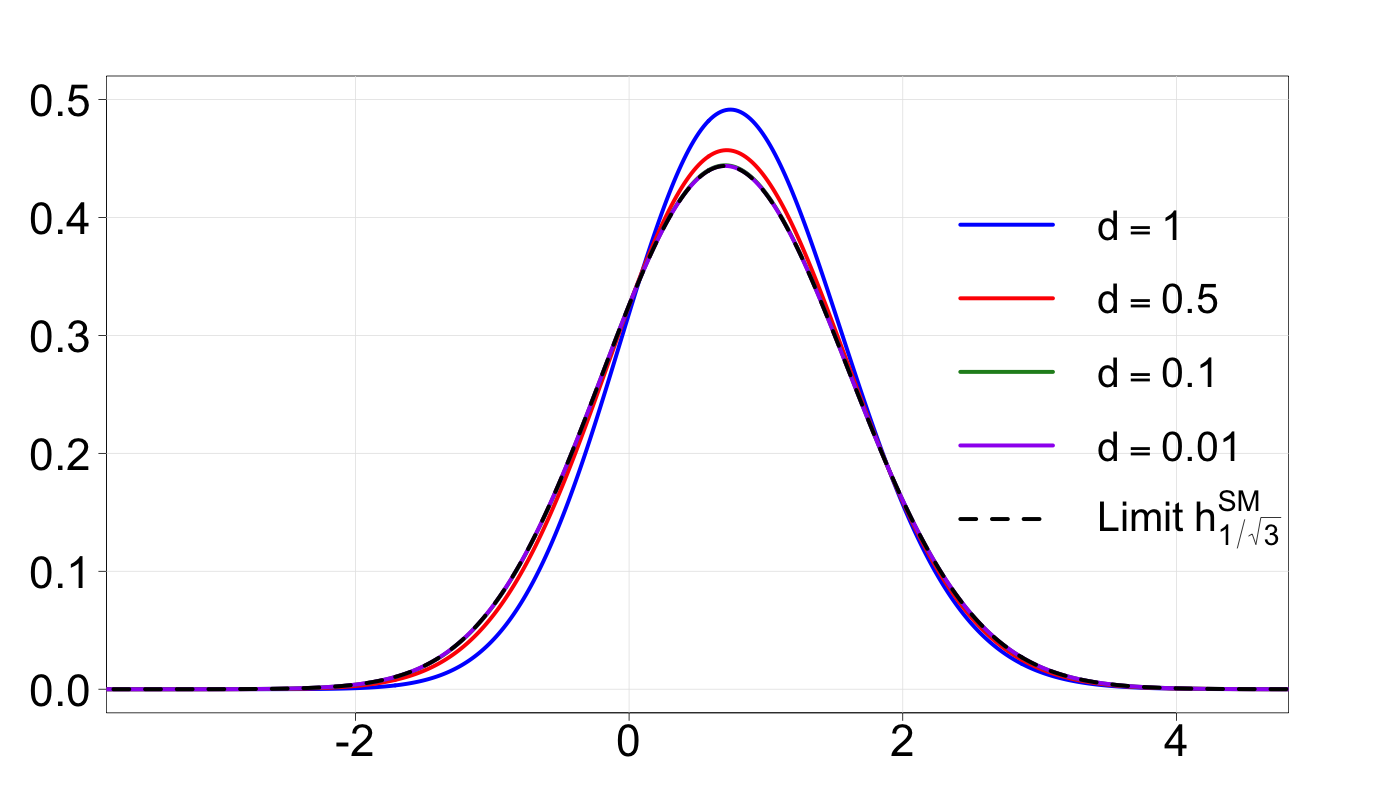}\\
			{\small (c) Mat\'ern $\nu=5/2$} &
			{\small (d) Squared exponential} 
		\end{tabular}
		\caption{\label{fig:density-comparison}\small
			Exact finite-grid peak height densities at mesh sizes $d=1, \, 0.5, \, 0.1, \, 0.01$, with their limiting densities. The Ornstein--Uhlenbeck and Mat\'ern $\nu=2$ limits are $\phi$; the Mat\'ern $\nu=5/2$ and squared-exponential limits are the smooth densities with $\eta=1/3$ and $\eta=1/\sqrt3$, respectively. All curves are obtained by deterministic numerical evaluations.}
	\end{center}
\end{figure}

\begin{figure}[tbp]
	\centering
	\includegraphics[width=.88\textwidth]{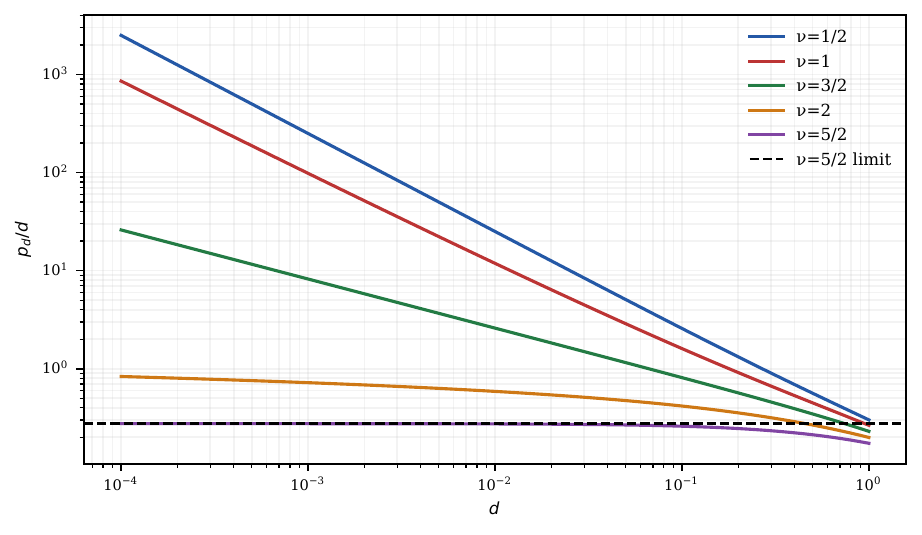}
	\vspace{-3mm}
	\caption{\small Exact Mat\'ern peak intensities per unit length, $p_d/d$,
		with $\theta=1$. On an interval of length $T$, the expected number
		of interior grid peaks is $Tp_d/d+O(p_d)$. The dashed line is the
		finite limit $\sqrt3/(2\pi)$ for $\nu=5/2$; the curves for
		$\nu\le2$ diverge as $d\downarrow0$.}
	\label{fig:matern-counts}
\end{figure}

\subsection{Deterministic numerical comparisons}\label{sec:numerics}

For all numerical comparisons we set $\lambda=1$ in the
Ornstein--Uhlenbeck covariance, $\theta=1$ in the Mat\'ern covariance,
and $\ell=1$ in the squared-exponential covariance. Thus the mesh values
in the density figure and the error table use the same model scales. The peak probability $p_d$ and finite-mesh peak height density $h_d(x)$ are computed from \eqref{eq:h-stationary}.

Figure~\ref{fig:density-comparison} compares finite-mesh densities with
their limiting laws. Figure~\ref{fig:matern-counts} shows the corresponding
Mat\'ern peak intensity per unit length. The distinction between
convergence of the entire density and relative accuracy of a small tail
probability is quantified by
\begin{equation*}
	E_1(d)=\int_{\R}|h_d(x)-h_\infty(x)|\dd x,\qquad
	R_u(d)=\frac{\overline H_d(u)}{\overline H_\infty(u)}-1,
\end{equation*}
where $h_\infty=\phi$ in the rough cases and $h_\infty=h^{\mathrm{SM}}$
in the smooth cases. Positive $R_u$ means that the limiting law
underestimates the exact upper tail.

Table~\ref{tab:numerical-errors} illustrates the slow boundary behavior.
At $d=0.01$, the Mat\'ern $\nu=2$ height-curvature correlation is about
$0.191$, but the Gaussian tail at $u=3$ is smaller than the exact tail
by a factor of about $1.97$. The squared-exponential smooth approximation
at the same mesh is much more accurate. These comparisons show that the accuracy of a limiting approximation depends on the covariance model.
When accurate tail probabilities are needed, the exact discrete distribution
at the observed grid spacing is therefore preferable. The table reports errors and correlations to at most three significant digits.

\begin{table}[tbp]
	\centering\small
	\caption{Errors of the limiting density and tails at the stated mesh. Here $E_1(d)=\int_{\mathbb R}|h_d(x)-h_\infty(x)|\dd x$ and $R_u(d)=\overline H_d(u)/\overline H_\infty(u)-1$. }
	\label{tab:numerical-errors}
	\begin{tabular}{@{}llllll@{}}
		\toprule
		Model & $d$ & $\eta_d$ & $E_1(d)$ & $R_2(d)$ & $R_3(d)$\\
		\midrule
		Ornstein--Uhlenbeck & 0.1 & $0.301$ & $0.273$ & $0.913$ & $1.29$ \\
		Ornstein--Uhlenbeck & 0.01 & $0.0995$ & $0.0896$ & $0.281$ & $0.397$ \\
		Squared exponential & 0.1 & $0.578$ & $0.000631$ & $-0.000926$ & $-0.00292$ \\
		Squared exponential & 0.01 & $0.577$ & $6.32\times10^{-6}$ & $-9.29\times10^{-6}$ & $-2.94\times10^{-5}$ \\
		Mat\'ern $\nu=2$ & 0.1 & $0.27$ & $0.27$ & $0.968$ & $1.44$ \\
		Mat\'ern $\nu=2$ & 0.01 & $0.191$ & $0.192$ & $0.656$ & $0.965$ \\
		Mat\'ern $\nu=5/2$ & 0.1 & $0.354$ & $0.0209$ & $0.0402$ & $0.0482$ \\
		Mat\'ern $\nu=5/2$ & 0.01 & $0.335$ & $0.00215$ & $0.00415$ & $0.00503$ \\		
		\bottomrule
	\end{tabular}
\end{table}

\section{Discussion}\label{sec:discussion}

The exact discrete law and the continuous peak law describe related
sampling rules. The discrete law is an ordinary conditional distribution
determined by three Gaussian observations. Under the smooth covariance
and nondegeneracy assumptions of Section~\ref{sec:smooth}, its infill
limit is $h_{\eta_t}^{\mathrm{SM}}$ with $\eta_t\in(0,1]$.
An interpretation as the Palm law of continuous local maxima requires
the additional sample path and Kac--Rice assumptions stated there.
Theorem~\ref{thm:rough-nonstationary} characterizes convergence to the
Gaussian law for general nonstationary covariance, including vanishing
peak probabilities.

Remark~\ref{remark:unify} places these established limits in the same
continuous family: $h_0^{\mathrm{SM}}=\phi$, the interior members are
the smooth peak laws, and $h_1^{\mathrm{SM}}$ is Rayleigh. For each fixed
nonconstant, mean-square continuous stationary process, the fourth
spectral moment identifies the parameter: $\eta=0$ for $m_4=\infty$
and $\eta=m_2/\sqrt{m_4}\in(0,1]$ for $m_4<\infty$.
This connects Gaussian height neutrality and smooth curvature weighting
within a single family, while retaining the assumptions that determine
which limit applies.

The height laws converge in total variation, which implies
\[
\sup_{u\in\R}|\overline H_d(u)-\overline H_\infty(u)|
\leq\frac12\int_\R|h_d(x)-h_\infty(x)|\dd x\to 0.
\]
This is uniform absolute tail convergence. It does not alone imply that
the ratio of the tails tends to one at thresholds $u_d\to\infty$.
The stronger density bound in Remark~\ref{rem:rough-relative-tails}
does give relative Gaussian tail convergence when
$u_d\sqrt{R_{t,d}}\to0$, with $R_{t,d}\to0$. In the stationary case
this sufficient condition is $u_d\eta_d\to0$.
Outside such controlled ranges, the distinction remains essential.
The following consequence of the Ornstein--Uhlenbeck formula makes this
distinction explicit. Fix $c>0$ and let $u_d=c/b_d$. Since $b_d\to0$,
$u_d\to\infty$, and
\begin{equation}\label{eq:OU-moving-tail}
	\frac{\overline H_d(u_d)}{\Psi(u_d)}
	=\frac{1}{p_d}\E\{\Phi(b_d Z)^2\mid Z>u_d\}
	\to 4\Phi(c)^2,
\end{equation}
where $Z$ is standard Gaussian. Indeed, for each $\varepsilon>0$,
\[
\P\{b_dZ>c+\varepsilon\mid Z>c/b_d\}
=\frac{\Psi((c+\varepsilon)/b_d)}{\Psi(c/b_d)}\to 0
\]
by the Gaussian tail asymptotic. Thus $b_dZ\to c$ in conditional
probability, and bounded continuity gives the expectation limit in
\eqref{eq:OU-moving-tail}; finally $p_d\to1/4$.
For this covariance,
\[
u_d\eta_d
=c\sqrt{\frac{2\{1+r(d)\}}{3-r(d)}}\to\sqrt2\,c>0.
\]
These thresholds therefore lie outside the sufficient range
$u_d\eta_d\to0$ for relative Gaussian tail convergence. Relative
approximation outside that range requires further tail analysis or
direct use of the exact finite-grid law.

Extensions to nonzero mean, varying variance, irregular observation grids, and Gaussian fields on $N$-dimensional lattices remain natural. Finite-dimensional Gaussian conditioning still yields exact comparison
probabilities in these settings. The limiting peak height law may,
however, change with the process model and the geometry of the comparisons.

\section*{Acknowledgments}
The authors were supported by NSF Grant DMS-2220523 and Simons Foundation Collaboration Grant \#854127.

	\begin{small}
		
	\end{small}
	
	\bigskip
	
	\begin{quote}
		\begin{small}
			
			\textsc{Dan Cheng and John Ginos}\\
			School of Mathematical and Statistical Sciences \\
			Arizona State University\\
			900 S. Palm Walk\\
			Tempe, AZ 85281, U.S.A.\\
			E-mail: \texttt{cheng.stats@gmail.com}

		\end{small}
	\end{quote}
	
\end{document}